\documentclass[reqno]{amsart}

\usepackage{amsmath,amssymb,amsthm,harpoon,accents}

\usepackage[foot]{amsaddr} 
\usepackage{mathrsfs} 
\usepackage{euscript}   
\usepackage{url} 
\usepackage{hyperref} 
\usepackage{enumitem} 
\usepackage{color} 
\usepackage{bbm} 

\hypersetup{colorlinks=true, citecolor=darkgreen, linkcolor=darkblue}
\definecolor{darkgreen}{rgb}{0.173, 0.706, 0.173}
\definecolor{darkblue}{rgb}{0.0,0,0.7}

\usepackage[frak=esstix]{mathalfa} 

\usepackage{tikz} 
\usepackage{tikz-cd} 
\usetikzlibrary{arrows} 
\usetikzlibrary{decorations.markings} 
\usepackage{standalone} 

\let\originalleft\left
\let\originalright\right
\renewcommand{\left}{\mathopen{}\mathclose\bgroup\originalleft}
\renewcommand{\right}{\aftergroup\egroup\originalright}

\makeatletter 
\newcommand{\overrightharpoon}{%
  \mathpalette{\overarrow@\rightharpoonfill@}}
\def\rightharpoonfill@{\arrowfill@\relbar\relbar\rightharpoonup}
\makeatother

\makeatletter 
\newcommand{\overrightsmallharpoon}{\mathpalette{\overarrowsmall@\rightharpoonfill@}}
\def\rightharpoonfill@{\arrowfill@\relbar\relbar\rightharpoonup}
\newcommand{\overarrowsmall@}[3]{%
  \vbox{%
    \ialign{%
      ##\crcr
      #1{\smaller@style{#2}}\crcr
      \noalign{\nointerlineskip}%
      $\m@th\hfil#2#3\hfil$\crcr
    }%
  }%
}
\def\smaller@style#1{%
  \ifx#1\displaystyle\scriptstyle\else
    \ifx#1\textstyle\scriptstyle\else
      \scriptscriptstyle
    \fi
  \fi
}
\makeatother

\newcommand{\arrow}[1]{\overrightsmallharpoon{#1}}

\newcommand{\Ascr}{\mathscr{A}} 

\newcommand{\Bl}{\arrow{{B}}} 
\newcommand{\Bscr}{\mathscr{B}} 
\newcommand{\Bscrl}{\arrow{\mathscr{B}}} 

\newcommand{\Cscrl}{\arrow{\mathscr{C}}} 

\newcommand{\eb}{\mathbf{e}} 
\newcommand{\El}{\arrow{E}} 

\newcommand{\Fl}{\arrow{{F}}} 
\newcommand{\Fscr}{\mathscr{F}} 
\newcommand{\Fscrl}{\arrow{\mathscr{F}}} 

\newcommand{\Hl}{\arrow{H}} 

\DeclareMathOperator{\id}{\mathnormal{o}} 

\DeclareMathOperator{\Kl}{\arrow{K}} 

\DeclareMathOperator{\LE}{{\mathsf{LE}}} 

\newcommand{\N}{\mathbb{N}} 
\newcommand{\Nbh}{{N}} 
\DeclareMathOperator{\nol}{\mathbf{0}} 

\newcommand{\R}{\mathbb{R}}
\newcommand{\rel}{\widehat{\rho}} 
\newcommand{\Rot}{\text{Rot}} 

\newcommand{\satubb}{\mathbbm{1}} 
\newcommand{\Sc}{\mathcal{S}} 

\newcommand{\Tl}{\arrow{{T}}} 

\newcommand{\Ul}{\arrow{{U}}} 

\DeclareMathOperator{\wsf}{\mathsf{{WSF}}} 

\newcommand{\xb}{\mathbf{x}} 

\newcommand{\yb}{\mathbf{y}} 

\newcommand{\Z}{\mathbb{Z}} 

\newcommand{\E}[1]{\mathbb{E}\left[ #1 \right]}

\newcommand{\Cond}[2]{\mathbb{E}\left[#1\mid #2\right]}

\newtheorem{theorem}{Theorem}[section]
\newtheorem{lemma}[theorem]{Lemma}
\newtheorem{proposition}[theorem]{Proposition}

\setlist[enumerate]{
label=\textnormal{({\roman*})},
ref={\roman*}}

\makeatletter
\def\th@plain{%
  \thm@notefont{}
  \itshape 
}
\def\th@definition{%
  \thm@notefont{}
  \normalfont 
}
\makeatother

\theoremstyle{definition}

\newtheorem{problem}{Problem}[section]
\newtheorem{remark}[theorem]{Remark}

\newtheorem{example}[theorem]{Example}

   \newtheorem{definition}[theorem]{Definition}

\makeatletter 
\newtheorem*{rep@theorem}{\rep@title}
\newcommand{\newreptheorem}[2]{%
\newenvironment{rep#1}[1]{%
 \def\rep@title{#2 \ref{##1}}%
 \begin{rep@theorem}}%
 {\end{rep@theorem}}}
\makeatother

\newreptheorem{theorem}{Theorem} 
\newreptheorem{conjecture}{Conjecture} 
\newreptheorem{corollary}{Corollary}
\newreptheorem{lemma}{Lemma}
\newreptheorem{proposition}{Proposition}

\makeatletter
\newcommand{\mylabel}[2]{#2\def\@currentlabel{#2}\label{#1}}
\makeatother

\newcounter{numberstationarity} 
\newcounter{numberergodic} 
\newcounter{numberscalinglimit} 
\newcounter{numberprops} 
\begin{document} 
\title{Random walks with local memory}

\author{Swee Hong Chan}
\address[Swee Hong Chan]{Department of Mathematics, UCLA, Los Angeles CA, corresponding author.}
\email{\texttt{sweehong@math.ucla.edu}}

\author{Lila Greco}
\address[Lila Greco]{Berkshire Hathaway Specialty Insurance, Boston MA.}
\email{\texttt{ecg83@cornell.edu}}

\author{Lionel Levine}
\address[Lionel Levine]{Department of Mathematics, Cornell University, Ithaca NY.}
\email{\texttt{levine@math.cornell.edu}}

\author{Peter Li}
\address[Peter Li]{Department of Economics, New York University, New York NY.}
\email{\texttt{bl2403@nyu.edu}}

\thanks{Lila Greco was partially supported by NSF grant DGE-1650441, Lionel Levine was partially supported by NSF grant DMS-1455272.}

 \date{\today}


\begin{abstract}
We prove a quenched invariance principle for a class of random walks in random environment on $\mathbb{Z}^d$, where the walker alters its own environment.
The environment consists of an outgoing edge from each vertex. The walker updates the edge $e$ at its current location to a new random edge $e'$ (whose law depends on $e$) and then steps to the other endpoint of $e'$.
We show that a native environment for these walks (i.e., an environment that is stationary in time from the perspective of the walker) consists of the wired uniform spanning forest oriented toward the walker, plus an independent outgoing edge from the walker. 
\end{abstract}



\maketitle

\section{A random environment altered by the walker}\label{section: intro}

Label each site of $\Z^2$ with either `H' or `V'.  A walker starts at the origin.  At each discrete time step the walker resamples the label at its current location (changing `H' to `V' and `V' to `H' with probability $q$, independent of the past) and then takes a mean zero horizontal step if the new label is `H' and a mean zero vertical step if the new label is `V'.  
We will show (see Theorem~\ref{theorem: intro scaling limit} below) that, for a certain distribution on initial labels, the scaling limit of the walk is a standard planar Brownian motion.

The walk just described is an example of a \emph{random walk with local memory} on a graph $G$.
 Each vertex of $G$ stores one bit of information in its label. For each vertex $x$ that the walk visits, the label of $x$ remembers whether the most recently traversed outgoing edge from $x$ was horizontal or vertical.  This memory in turn affects the distribution of the edge traversed the next time the walker returns to $x$.  One can consider more complicated forms of local memory (e.g., that remember several past visits) but they all essentially reduce to the standard \emph{retrospective form}, i.e.,  each vertex $x$ is labeled by an outgoing edge from $x$ (see Appendix~\ref{section: hidden random walks} for the reduction).  At each discrete time step, the walker updates the label $e$ of its current location to a new random edge $e'$ (whose law depends on $e$) and then steps to the other endpoint of $e'$.

Pinsky and Travers~\cite{PT17} and  Kosygina and Peterson~\cite{KP17} study random walks with local memory  in one dimension under the name ``Markovian cookie stacks'',
where the labels evolve following the transition rules of a prescribed Markov chain for each each vertex.
(These Markov chains are assumed to be independent; see Travers~\cite{Tra18} for the case when they are not independent.) 
In particular, the latter  characterizes when such a walk is recurrent, transient non-ballistic, or ballistic;
and  derives a central limit theorem for the transient case.
The methods used in \cite{KP17} are based on the generalized Ray-Knight theory developed by T\'{o}th~(see \cite{Tot99} and references therein) for generalized reinforced/repelling random walks, and are limited only to one dimension. 
 The aim of this paper is to begin the study of these walks in higher dimensions, by identifying a native environment and proving an invariance principle.



In analyzing random walks with local memory in higher dimensions, we take our inspiration from the theory of random walk in random environment~\cite{Zei04,Szn04}, in which the environment affects the motion of the walker but the walker does not affect the environment. 
 In our walks, a new difficulty is that the walker alters its own environment.


\subsection{Main results}

An interesting feature of random walk with local memory is that the walker organizes its environment to form a tree. Indeed, when the walk is expressed in retrospective form, the local state at each previously visited vertex is the last exit edge, so the edges at visited vertices form a tree oriented toward the walker.  From this observation, it is natural to use the wired  spanning forest (defined below) to construct  a \emph{native environment}~(i.e., an invariant measure for the environment viewed from the perspective of the walker; see Definition~\ref{definition: native environment}).

Let $G$ be a simple connected graph that is locally finite (i.e., each vertex has a finite degree). 
Let  $W_1\subseteq W_2 \subseteq \ldots$ be finite connected subsets of $V(G)$  such that $\bigcup_{n=1}^\infty W_n=V(G)$.
Let $G_n$ be obtained from $G$ by identifying all vertices outside $W_n$ to one new vertex, and let $\mu_n$ be the uniform measure on spanning trees of $G_n$.
The \emph{wired uniform spanning forest}, denoted by  $\mathsf{WUSF}$, is then the unique infinite-volume limit of  $\mu_n$.

Fix a vertex $o$ of $G$ as the initial location of the walker.
To build an initial environment from $\mathsf{WUSF}$, we orient the connected component of $o$ in the $\mathsf{WUSF}$ toward  $o$, orient all other components toward infinity, and add an independent outgoing edge from the $o$.
Note that there might be more than one way to orient a component toward infinity if it has more than one end;
we will orient them using the orientation given by Wilson's method rooted toward infinity~\cite{BLPS01}.
We denote by  $\arrow{\mathsf{WUSF}}^+$ the resulting initial environment.
This environment is a {native environment}
under the following assumptions.

Assume that $G$ is a simple (undirected) Cayley graph of a finitely generated group.
A random walk with localy memory is \emph{transitive} if every vertex follows the same rule in updating its local memory; see~\eqref{equation: transitive}.
A random walk with local memory is \emph{uniform} if,  averaging over all  initial labels, every outgoing edge of the current location is equally likely to be the next label.
We remark that we actually prove the main results under a weaker uniformity assumption
called \emph{$c$-stationarity}~(see~\eqref{equation: c-stationary}),
and we only use the uniformity assumption in this section to simplify the notation. 

\smallskip

\begin{theorem}\label{theorem: intro stationarity}
Consider a random walk with local memory  on a simple Cayley graph that is transitive and uniform.
Then $\arrow{\mathsf{WUSF}}^+$ is a native environment.
\end{theorem}
\setcounter{numberstationarity}{\value{theorem}} 

\smallskip

We prove  Theorem~\ref{theorem: intro stationarity} in \S\ref{stationarity}   by  proving an analogous statement for finite graphs, and then passing to a limit.

It turns out that $\arrow{\mathsf{WUSF}}^+$ satisfies a stronger property, namely that it is an \emph{ergodic native environment} (i.e., an ergodic measure for the environment viewed from the perspective of the walker; see Definition~\ref{definition: ergodic environment}), under the additional assumption that the random walk with local memory is \emph{elliptic} (i.e., every neighbor of the current location is visited next with positive probability; see~\eqref{equation: elliptic}).

\smallskip

\begin{theorem}\label{theorem: intro ergodic}
	Consider a random walk with local memory  on a simple Cayley graph that is transitive, uniform, and elliptic.
	Then $\arrow{\mathsf{WUSF}}^+$ is an ergodic native environment.
\end{theorem}
\setcounter{numberergodic}{\value{theorem}} 

\smallskip

We prove  Theorem~\ref{theorem: intro ergodic} in \S\ref{subsection: proof of ergodic native environment} 
through a delicate combinatorial argument that makes use of the tail triviality of $\mathsf{WUSF}$.
We believe that the ellipticity assumption is not necessary for the conclusion of Theorem~\ref{theorem: intro ergodic}; see \S\ref{subsection: dropping ellipticity}.

 Our next result is the following functional CLT for when
 $G$ is a \emph{lattice graph} in $\R^d$ (i.e., a Cayley graph such that $V(G)$ is a subgroup of $\R^d$ with vector addition as the group operation). 
A random walk with local memory is  a \emph{martingale} if,
conditioned on the present location and label,
the expected next location of the walker is equal to the present location; see~\eqref{equation: martingale}.
%

For every outgoing edge $e$ of the initial location $o$,
let $Y_e \in \R^d$ be the location of the walker after one step of the walk, assuming $e$ is the initial label at $o$.
We denote by $\Gamma_e$ the $d \times d$ covariance matrix $\E{Y_e Y_e^\top}$, and  
 by $\Gamma$ the average of covariance matrices of outgoing edges of $o$. 

\smallskip

\begin{theorem}\label{theorem: intro scaling limit}
Consider a random walk with local memory on a simple lattice graph in $\R^d$  that is transitive, uniform, and is a martingale.
Suppose that the initial environment is an ergodic native environment $\pi$.
Then, for   almost every environment sampled from $\pi$, 
the trajectory of the walker scales to a  Brownian motion in $\R^d$.
That is to say, 
\[  \frac{1}{\sqrt{n}}  ({X_{\lfloor nt \rfloor}})_{t \geq 0}  \  \overset{ n \to \infty}{ \Longrightarrow} \  B(t).\]
Here $X_{\lfloor nt \rfloor}$ is the location of the walker at the $\lfloor nt \rfloor$-th step of the walk,
 $B(t)$ is a   Brownian motion in $\R^d$ with diffusion matrix $\Gamma$, 
 and the convergence is weak convergence in the Skorohod space $D_{\R^d}[0,\infty)$.
\end{theorem}
\setcounter{numberscalinglimit}{\value{theorem}} 

\smallskip


In particular, Theorem~\ref{theorem: intro scaling limit} applies 
to the `H,V'-walk described in the beginning with $q$ strictly between $0$ and $1$.
We prove Theorem~\ref{theorem: intro scaling limit} 
in \S\ref{SLLN} by using  
  standard tools in random walks in random environments,
namely the martingale CLT and the pointwise ergodic theorem, and 
we illustrate the flavor with the `H,V'-walk  here. By the martingale CLT,
the problem reduces to showing that  the
walker encounters the label `V'  half of the time, i.e.,
\begin{equation}\label{equation: SLLN}
  \frac{1}{n} \sum_{i=0}^{n-1} \mathbbm{1}\{\text{the label  used by the walker at the $i$-th step is `V'}\} \longrightarrow \frac{1}{2},
  \end{equation}
in probability as $n \to \infty$.
The convergence in \eqref{equation: SLLN} in turn follows from the pointwise ergodic theorem.
Note that, in order to apply the pointwise ergodic theorem,
the initial environment needs to be native and ergodic,
and $\arrow{\mathsf{WUSF}}^+$ is such an environment by Theorem~\ref{theorem: intro ergodic}.

\smallskip

Our final result is the following functional CLT, assuming  a stronger regularity condition on the RWLM  but requiring no condition on the initial environment. 
An RWLM has \emph{identical local covariances} if $\Gamma_e=\Gamma_{e'}$ for every outgoing edge $e$ of $o$.

\smallskip

\begin{proposition}\label{prop: scaling limit}
	Consider a random walk with local memory on a simple lattice graph in $\R^d$  that is transitive,  is a martingale, and has identical local covariances.
	Then, for {every} initial environment, 
	\[  \frac{1}{\sqrt{n}} \ ({X_{\lfloor nt \rfloor}})_{t \geq 0}    \overset{ n \to \infty}{ \Longrightarrow} \  B(t),\]
	where $X_{\lfloor nt \rfloor}$ is the location of the walker at the $\lfloor nt \rfloor$-th step of the walk, and
	$B(t)$ is a   Brownian motion in $\R^d$ with diffusion matrix $\Gamma$.
\end{proposition}
\setcounter{numberprops}{\value{theorem}} 
\smallskip

 We prove Proposition~\ref{prop: scaling limit} (under slightly weaker assumptions) in \S\ref{scaling limit} as a direct application of the martingale CLT.
In particular, Proposition~\ref{prop: scaling limit} applies to the random walk with local memory on the triangular lattice
where the mechanism is rotating the current outgoing edge by 60 degrees, 180 degrees, or 300 degrees, each with probability $\frac{1}{3}$; see Example~\ref{ex: triangular lattice}.
On the other hand, Proposition~\ref{prop: scaling limit} does \emph{not} apply to `H,V'-walk if $q\neq \frac{1}{2}$ (since  
$\Gamma_e$ is equal to 
 $\begin{bmatrix}
 	1-q & 0 \\
 	0 & q
 \end{bmatrix} $ 
if $e$ is a horizontal edge, and is equal to 
$\begin{bmatrix}
	q & 0 \\
	0 & 1-q
\end{bmatrix}$ if $e$ is a vertical edge). 
This necessitates results such as Theorem~\ref{theorem: intro scaling limit} that has weaker assumptions and does apply to a family of models that include  
`H,V'-walk.

%

\medskip

\subsection{Other related work}\label{subsection: Other related work}
\subsubsection{} When each vertex uses a deterministic rule to update its local memory,
the random walk with local memory is known as \emph{rotor walk}  (discovered independently by \cite{WLB96, PDDK96, Pro03}).
In this model, each vertex is given a prescribed cyclic ordering on its outgoing edges,
and for every update the vertex changes the current edge to the next edge in the cyclic order.
A fundamental difficulty with rotor walk is its lack of randomness: 
For example, it is an open problem to prove that the rotor walk in $\Z^2$ with i.i.d.\ uniform initial rotors is recurrent; see \cite{HLM08, FLP16} for an exposition of this and related problems.

{\renewcommand{\arraystretch}{-10}
\begin{figure}[h!]
\begin{tabular}{c c c}
\hspace{-0.5 cm}
\includegraphics[scale=0.38]{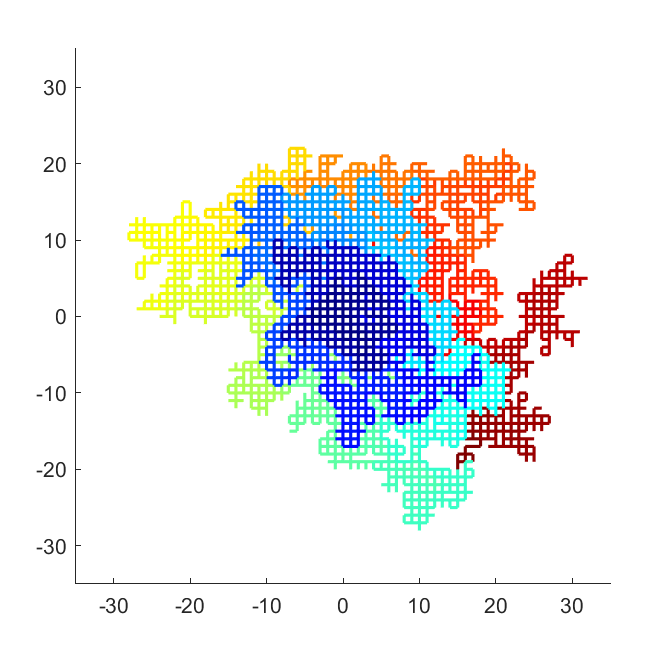} & \hspace{-0.25cm}  
\includegraphics[scale=0.38]{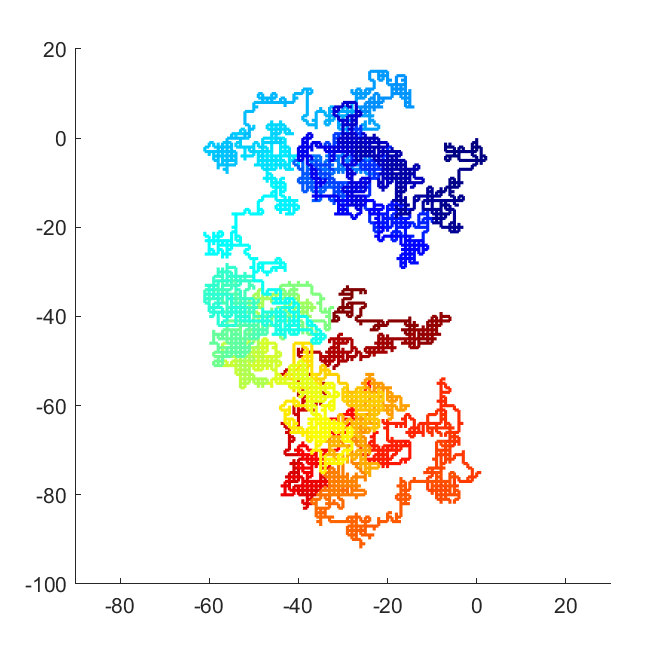} & \hspace{-0.25cm} 
\includegraphics[scale=0.38]{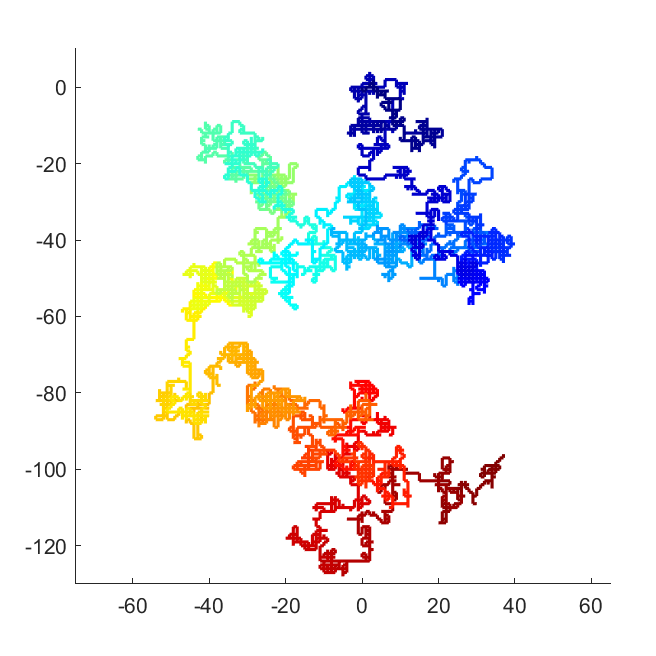}
\end{tabular}
\caption{The vertices  visited by  a 10,000-step  
   rotor walk (left),  `H'-`V' walk with $q=1$ (middle), and  simple random walk (right) on $\mathbb{Z}^2$;
   these processes are ordered in increasing amount of randomness.
Each edge is colored according to the time of its first visit by the walker. }
    \label{figure: simple random walk vs rotor walk}%
\end{figure}
}


%

\subsubsection{}  One dimensional random walk with local memories  are more commonly studied 
in the literature under the name \emph{excited random walks} (introduced by Benjamini and Wilson~\cite{BW03}): A pile of cookies is initially placed at each vertex of $\Z^d$ ($d\geq 1$).
Upon visiting  a vertex,
the walker consumes the topmost cookie from the pile  and moves to the neighboring vertex according to  probabilities prescribed by that cookie.
If there are no cookies left at the current vertex,
the walker chooses a neighbor uniformly at random and moves there.

The functional limit theorem for excited random walks on $\Z$ have been studied for the case of bounded number of i.i.d. cookies~\cite{KM11,DK12}, periodic cookies \cite{KP16}, and Markovian cookies~\cite{KP17,HLSH18}, among others. 
The functional limit theorem for higher-dimensional walks are much rarer in comparison.
Nevertheless, 
it has been studied for the case of a single cookie with  drift to a specific direction by  \cite{VdHH12} (for dimensions $d>8$ and a specific drift intensity), by  \cite{BR07} (for all dimensions), and by \cite{MPRV12} (for all dimensions under more general assumptions).
We refer the reader to \cite{KZ13} for an excellent survey on excited random walks.
Finally, in the direction of non-Markovian walks,
the most relevant recent work is \cite{BL19}, which applies  martingale theory to higher-dimensional elephant random walks.

The main motivation of this paper is to begin extending the results of \cite{KP17,HLSH18} from dimension one to higher dimensions, which we partly achieve in Theorem~\ref{theorem: intro scaling limit}. 
In particular, it is shown in \cite{HLSH18} that the scaling limit for $p$-rotor walk in $\Z$ 
(where the next edge points in the same direction as the current edge with probability $p$, and points in the opposite direction with probability $1-p$)
is a Brownian motion perturbed at extrema.
This perturbation is caused by the  initial environment in \cite{HLSH18} not being a native environment. 
We expect that proving a scaling limit 
for any higher-dimensional random walk with local memory in a non-native environment
 will require major new ideas (for example, what are the planar and higher-dimensional analogues of the one-dimensional Brownian motion perturbed at extrema?).

\subsubsection{} 
A \emph{self-interacting random walk}
(SIRW) is a nearest-neighbour walk on 
$\Z^d$,
where at each step the probability of the walker to jump along a certain direction $\alpha$ is proportional to $w(n_{\alpha})$,
where $w:\N  \to \R_{>0}$ is a monotone weight function and $n_{\alpha}$ is the number of previous jumps along the direction $\alpha$. 
Unlike random walks with local memory, the transition probabilities for SIRW  depend on all of the previous visits to the current location
rather than just the most recent visit.
Various  limit theorems for various one-dimensional SIRWs were studied by T\'{o}th~(see e.g., \cite{Tot95,Tot96}), and we refer to the survey~\cite{Tot99} for references on this subject.
It remains to be seen if the methods of this paper can be applied to SIRWs in higher dimensions.

\subsubsection{}  The idea of viewing the environment from the perspective of the walker dates back to the work of  Kozlov~\cite{KOZ85} and Papanicolaou-Varadhan~\cite{PV82}.
We refer the reader to~\cite[Lecture~1]{BS02} for references on this subject.

\subsubsection{}   Random walk with local memory is a special case of the \emph{stochastic abelian networks} defined in \cite{BL16}. 
More precisely, a random walk with local memory is a unary network in which every processor sends exactly one letter of output for each letter of input.  
From this perspective, a general stochastic abelian network can then be viewed as a branching random 
walk with local memory with multiple types of walkers.

\subsection{Outline}

In \S\ref{rotor walks} we give the rigorous definition of random walks with local memory.
In \S\ref{scaling limit} we prove Proposition~\ref{prop: scaling limit}.
In \S\ref{ust} we construct the wired spanning forest  oriented toward a fixed vertex, which is a simple modification of the construction in \cite{BLPS01}. 
In \S\ref{stationarity} we use the oriented wired spanning forest from \S\ref{ust} to construct a native environment for random walk with local memory, and proves Theorem~\ref{theorem: intro stationarity}. 
In \S\ref{section: ergodic} we prove Theorem~\ref{theorem: intro ergodic}.
In \S\ref{SLLN} we prove Theorem~\ref{theorem: intro scaling limit}.
In \S\ref{section: further questions} we conclude with a list of open problems.
In Appendix~\ref{section: hidden random walks} we show the reduction that converts random  walks with more complicated forms of local memory to  the standard retrospective form, at the cost of  
changing the underlying graph to a larger graph that might  have multiple edges.

\bigskip

\section{Random walks with local memory}\label{rotor walks}

Throughout this paper $G:=(V(G),E(G))$ denotes a connected, undirected graph that is locally finite (every vertex has finite degree) and simple (no loops, no multiple edges).
We remark that 
all the results in this paper can be extended  to non-simple graphs verbatim;
and we simply restrict to the case of simple graphs to
simplify the  notation.
When the graph $G$ is evident from context, we will omit $G$ from the notation and  write $V$ and $E$ instead.

A {\em neighbor} of a vertex $x$ is a vertex $y$ such that $\{x,y\} \in E$. We denote by $\Nbh(x)$ the set of all neighbors of $x$. 
An \emph{oriented edge} of $G$ is a pair $(x,y) \in V \times V$ such that $\{x,y\}$ is an (unoriented) edge of $G$.
We call  $(x,y)$ an {\em outgoing} edge of $x$ and an {\em incoming} edge of $y$. 
In an oriented subgraph of $G$,
the \emph{outdegree} (respectively, \emph{indegree}) of $x$ is the number of outgoing (respectively, \emph{incoming}) edges of $x$ in the oriented subgraph.
We denote by $\El$ the set of oriented edges of $G$.
The running example for a graph in this paper is the integer lattice $\Z^d$ of dimension $d$, i.e., the graph given by
\begin{align*}
V&:=\{ \xb  \mid \xb \in \Z^d \}; \qquad {E}:=\{ \{\xb, \yb\} \in \Z^d \times \Z^d \mid ||\xb-\yb||=1  \},
\end{align*}
where $|| \cdot ||$ denotes the Euclidean norm.

\smallskip

\begin{definition}[Mechanism]\label{definition: mechanism}
A \emph{mechanism} of a random walk with local memory
is a collection of independent Markov chains $\{M_x\}_{x \in V}$
indexed by the vertices of $G$,
such that  the state space of $M_x$ is $N(x)$, the set of neighbors of $x$.
We denote   by $p_x(\cdot,\cdot)$ the \emph{probability transition function}  of the chain $M_x$.
\end{definition}

\smallskip

A {\em rotor configuration} of  $G$ is a map $\rho:V \to V $ such that $\rho(x)$ is a neighbor of $x$ for all $x \in V$.
This should be thought of as assigning to each vertex $x$ of $G$ a \emph{rotor} which points to a neighbor of $x$ via an oriented edge of $G$. 
A \emph{walker-and-rotor configuration} is a pair $(x,\rho)$, where $x$ is a vertex of $G$ and $\rho$ is a rotor configuration of $G$.

\smallskip

\begin{remark}\label{remark: rotor configurations dual}
A 
rotor configuration can be interpreted as either:
\begin{itemize}
\item A function $\rho: V \to V$ such that $\rho(x) \in \Nbh(x)$ for all $x\in V$; or
\item An oriented  subgraph of $G$ that has exactly one outgoing edge of each vertex of $G$.
\end{itemize}
These two objects are identified with each other by the map 
$\rho \mapsto (V(\rho), E(\rho))$, where 
\begin{align*}
  V(\rho):=V, \qquad  E(\rho):=\{(x,\rho(x)) \mid x \in V\}.
\end{align*}
We would like to warn the reader that both interpretations will be used  interchangeably starting from \S\ref{stationarity}. 
\end{remark}

\smallskip

\begin{definition}[Random walk with local memory]\label{definition: random walk with local memory}
A \emph{random walk with local memory}, or RWLM for short, is a sequence  
$(X_n,\rho_n)_{n \geq 0}$  of  walker-and-rotor configurations satisfying the following transition rules:
\begin{equation}\label{equation: transition rule RWLM}
\begin{split}
\rho_{n+1}(x):=&\begin{cases}  Y_{n} & \text{if } x=X_n;\\
\rho_{n}(x) & \text{if } x\neq X_n. \end{cases}; \text{ and }\\
X_{n+1}:=&Y_{n},
\end{split}
\end{equation}
where $Y_{n}$ is a random neighbor of $X_n$ sampled from $p_{X_n}(\rho_n(X_n),\cdot)$ 
independent of the past.
\end{definition}

\smallskip

Described in words,  $X_n$ records the location of the walker and $\rho_n$ records the rotor configuration at time  $n$ of the RWLM.
At time $n$, 
the walker updates 
  the rotor of $X_n$
using the Markov chain  $M_{X_n}$ (which depends only on $X_n$ and $\rho_n(X_n)$),
and then moves to the vertex to which the new rotor is pointing. 
The \emph{local memory} in the name refers to the fact that the walker records  the last exit from each vertex that it visits via the rotor configuration.
See Figure~\ref{figure: p-rotor walk}
for an illustration of an RWLM  on $\Z^2$. 

\begin{figure}[ht]
\begin{tabular}{c @{\hskip 30 pt} c @{\hskip 30 pt} c @{\hskip 30 pt} c}
\includegraphics[scale=0.7]{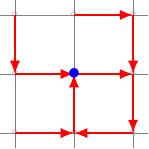} & 
\includegraphics[scale=0.7]{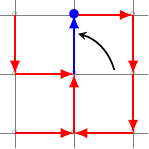} &
\includegraphics[scale=0.7]{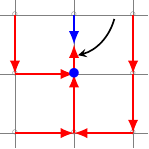} &
\includegraphics[scale=0.7]{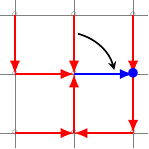} \\
$(X_0,\rho_0)$ & $(X_1,\rho_1)$ & $(X_2,\rho_2)$ & $(X_3,\rho_3)$
\end{tabular}
\caption{Three steps of a  random walk with local memory on $\Z^2$.
The location of the walker is given by {\color{blue} $\bullet$}, and  the rotor of each vertex is given by the arrow pointing out from the vertex.}
\label{figure: p-rotor walk}
\end{figure}

Naturally, the dynamics of the RWLM depend on the choice of 
the mechanism.
The following are three examples of RWLMs that have appeared in the literature:
\begin{enumerate}
\item Aldous-Broder walk, in which
the walker performs a simple random walk on $G$ and the rotor configuration never influences the decision of the walker.
That is to say, 
 for every $x \in V$ and $y \in \Nbh(x)$ the measure  $p_x(y, \cdot)$ is the  uniform distribution on  the neighbors of $x$.
 Our name for this walk comes from the algorithm of Aldous \cite{Ald90} and Broder \cite{Bro89} that generates the uniform spanning tree as a tree of first entrances of this walk.

\item Rotor walk \cite{WLB96,PDDK96,Pro03},
in which the Markov chain $M_x$ is given by  
 a deterministic permutation $\tau_x$ of the neighbors of $x$.
 That is, the  chain $M_x$ in state $y$ will transition to $\tau_x(y)$ with probability 1.
 We refer to \cite{HLM08,FLP16} for more details.

\item $p$-rotor walk on $\Z$~\cite{HLSH18} for $p\in[0,1]$, in which the probability transition function $p_x$ $(x \in \Z)$ is 
given by 
\[p_x(x\pm  1, x\mp1)=1-p; \qquad  p_x(x\pm  1, x\pm 1)=p. \]
\end{enumerate}

\smallskip

We now present three other examples of RWLMs.

\smallskip

\begin{example}[$p$-rotor walk on $\Z^d$]\label{ex: p-rotor walk}
Fix $d\geq 2$ and $p\in[0,1]$.
Denote by $\eb_1,\ldots, \eb_d$ the canonical basis of $\R^d$.
The Markov chain $M_{\xb}$ ($\xb \in \Z^d$) has state space $\{ \xb \pm \eb_i \mid 1\leq i \leq d  \}$ and  has the following  transition rule:
\[ \xb\pm \eb_i \quad \text{ transitions to } \quad 
\begin{cases} \xb\pm \eb_j & \text{with probability $\frac{p}{d-1}$ if $i<j$}; \\
\xb\mp \eb_j & \text{with probability $\frac{1-p}{d-1}$ if $i<j$};\\
\xb\pm \eb_j & \text{with probability $\frac{1-p}{d-1}$ if $i>j$};\\
\xb\mp \eb_j & \text{with probability $\frac{p}{d-1}$ if $i>j$}.   
\end{cases}    \]
Described in words, if the rotor at the particle's current location is parallel to $\eb_i$, 
the walker first picks $j$ uniformly from $\{1,\ldots, d\} \setminus \{i\}$.
Then, the walker rotates the current rotor counterclockwise 
in the $\{\min(i,j),\max(i,j)\}$-plane with probability $p$, and rotates clockwise 
with probability $1-p$.
See  Figure~\ref{figure: rotation diagram p-rotor walk} for an illustration of this mechanism on $\Z^2$.
\begin{figure}[t!]
\begin{tabular}{c @{\hskip 40 pt} c @{\hskip 40 pt} c @{\hskip 40 pt} c}
\includegraphics[scale=0.6]{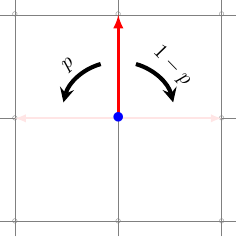} &\includegraphics[scale=0.6]{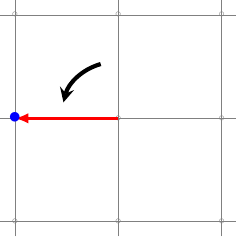} & \includegraphics[scale=0.6]{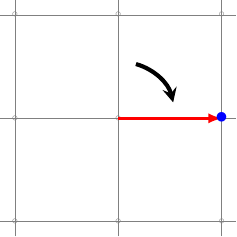}\\
(a) & (b) & (c)
\end{tabular}
\caption{(a) The mechanism for $p$-rotor walk on $\Z^2$, in which the rotor rotates counterclockwise with probability $p$, and clockwise with probability $1-p$.
  The location  of the walker and the rotor after one step of the RWLM
is given by (b)   
   if the walker chooses to rotate the rotor counterclockwise, and by (c) if the walker  chooses to rotate the rotor clockwise.}
\label{figure: rotation diagram p-rotor walk}
\end{figure}
\end{example}

\smallskip

\begin{example}[$p,\!r$-rotor walk on $\Z^d$]\label{example: p,r-rotor walk}
Fix $d\geq 2$, $p\in [0,1]$, and  $r\in [0,1]$.
For each visit to $\xb \in \Z^d$,
the mechanism at $\xb$ 
transitions according to the mechanism of Aldous-Broder walk with probability $1-r$,
and transitions according to the mechanism of $p$-rotor walk with probability $r$, independent of the past visits.
Note that we recover `H,V'-walk on $\Z^2$ for $q\leq \frac{1}{2}$ in \S\ref{section: intro}
by taking  $p=\frac{1}{2}$ and $r=1-2q$.
Also note that,  unlike  $p$-rotor walks,
in this model
every neighbor of the current location of the walker (all $2d$ of them) is visited  next  with positive probability provided that $r<1$ (i.e., the walk is {elliptic}).
See  Figure~\ref{figure: rotation diagram p,r-rotor walk} for an illustration of this mechanism.
\begin{figure}[t!]
\centering 
\includegraphics[scale=0.65]{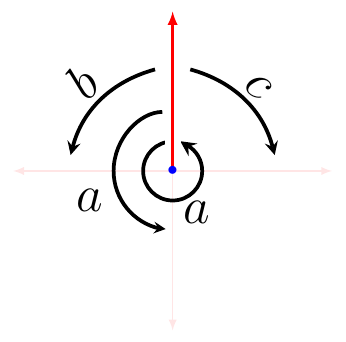}
\caption{The mechanism for  $p,\!r$-rotor walk on $\Z^2$, which stays at the current rotor with probability $a:=\frac{1-r}{4}$, rotates 180 degrees with probability $a$,
rotates 90 degrees counterclockwise with probability $b:=\frac{1-r}{4}+pr$, and rotates 90 degrees clockwise with probability $c:=\frac{1-r}{4}+(1-p)r$.}
\label{figure: rotation diagram p,r-rotor walk}
\end{figure}
\end{example}

\smallskip

\begin{example}[Triangular walk]\label{ex: triangular lattice}
The triangular lattice is the graph embedded in $\R^2$ given by:
\begin{align*}
V&:=\left \{ a \begin{pmatrix}
1\\0
\end{pmatrix} +  b\begin{pmatrix}
1/2 \\ \sqrt{3}/2 
\end{pmatrix}  \bigg | \,  a,b \in \Z \right \};\\
{E}&:=\{\{\xb,\yb\} \in V \times V \mid \|\xb-\yb\|=1\}.
\end{align*}
In this RWLM, the walker updates the current rotor  by applying a counterclockwise rotation by either 60 degrees, 180 degrees, or 300 degrees, 
each with probability $\frac{1}{3}$.
See Figure~\ref{figure: rotation diagram triangular walk} for an illustration of this mechanism.
\begin{figure}[t!]
\centering
\begin{tabular}{c @{\hskip 60 pt} c }
\includegraphics[scale=0.7]{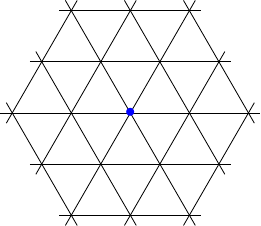} &\includegraphics[scale=0.7]{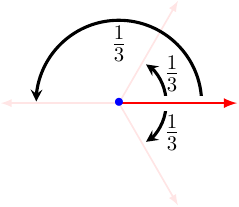} \\
(a) & (b) 
\end{tabular}
\caption{(a) The triangular lattice.
 (b) The mechanism for the triangular lattice, which rotates either 60 degrees counterclockwise, 180 degrees counterclockwise, or 300 degrees counterclockwise, 
each with probability $\frac{1}{3}$.}
 \label{figure: rotation diagram triangular walk}
\end{figure}
\end{example}

\bigskip

\section{Martingale central limit theorem} \label{scaling limit}

In this section we show that, under strong  regularity assumptions on the  RWLM,  
we can directly prove functional CLT from 
 the vector-valued martingale CLT proved in~\cite{RAS05}.
We denote by $D_{\R^d}[0,\infty)$ the Skorohod space  of $\R^d$-valued c\`{a}dl\`{a}g paths on $[0,\infty)$.
Recall that $|| \cdot ||$ denotes the Euclidean metric. 

\smallskip

\begin{theorem}[Martingale CLT~{\cite[Theorem~3]{RAS05}}]\label{CLT}
Let $(X_n)_{n \geq 0}$ be an $\R^d$-valued square-integrable martingale process w.r.t.\ a filtration  $(\Fscr_n)_{n \geq 0}$, 
and let $V_{n}:=X_{n+1}-X_{n}$ be the corresponding martingale difference sequence.
Suppose that:
\begin{enumerate}
\item There exists a symmetric, nonnegative definite $d \times d$ matrix $\Gamma$ such that
\[ \frac{1}{n}\sum_{i=0}^{n-1} \Cond{V_i V_i^\top}{\Fscr_i} \to
 \Gamma \qquad \text{in probability as } n \to \infty; \tag{CLT1} \label{CLT1}\] 
\item  For any $\epsilon>0$,
\[  \frac{1}{n}\sum_{i=0}^{n-1} \Cond{\|V_i\|^2 \, \satubb\{\|V_i\| \geq \epsilon \sqrt{n}\}}{\Fscr_i}  \to 0 \qquad \text{in probability as }n \to \infty.   \tag{CLT2} \label{CLT2}\]
\end{enumerate}
Then  
$ \left\{ \frac{1}{\sqrt{n}}{X_{\lfloor nt\rfloor}} , t\geq 0\right \} $
converges weakly on  $D_{\R^d}[0,\infty)$ to a Brownian motion with diffusion matrix $\Gamma$.
\qed
\end{theorem}

\smallskip


We now apply Theorem~\ref{CLT} to RWLMs under the following assumptions.
Let $G=(V,E)$ be a simple connected graph such that $V$ is a subset of $\R^d$.
An RWLM is \emph{bounded} if 
\begin{equation}\label{equation: BDD}
 \sup_{ \{\xb,\yb\}\in E} ||\xb-\yb||<\infty; \tag{Bdd} 
\end{equation}
All the RWLMs described in \S\ref{rotor walks} are bounded.

Recall the definition of probability transition functions $p_{\xb}$ from Definition~\ref{definition: mechanism}.
Let $\xb$ be a vertex of $G$, and let $\yb$ be a neighbor of $\xb$.
We denote by $Y_{\xb,\yb}$ the random variable sampled from 
$p_{\xb}(\yb,\cdot)$.
The \emph{local covariance matrix} of $\xb,\yb$ is the $d \times d$ matrix  $\Gamma_{\xb,\yb} := \E{(Y-\xb)(Y-\xb)^\top}$.

We say that an RWLM is a \emph{martingale} if
\begin{equation}\label{equation: martingale}
	\E{Y_{\xb,\yb}} \ = \ \xb   \qquad \text{for every $\xb \in V$ and $\yb \in N(\xb)$}. \tag{Mtgl}  
\end{equation}
Note that this condition is equivalent to  requiring  the sequence  $(X_n)_{n\geq 0}$ of locations of walker of the RWLM  to be  a martingale.
The  Aldous-Broder walk on $\Z^d$ and the triangular walk (Example~\ref{ex: triangular lattice}) is a martingale, the deterministic rotor walk is \emph{not} a martingale,  
and the $p$-rotor walk (Example~\ref{ex: p-rotor walk}) and 
$p,\!r$-rotor walk (Example~\ref{example: p,r-rotor walk})
are martingales only if $p=\frac{1}{2}$.

We say that an RWLM has \emph{identical local covariances} if
\begin{equation}\label{equation: ILC}
	\Gamma_{\xb,\yb} \ = \ \Gamma_{\xb',\yb'}  \qquad \text{for every $x,x' \in V$ and $\yb \in N(\xb), \yb' \in N(\xb')$,} \tag{ILC}
\end{equation}
and in this case we write $\Gamma:= \Gamma_{\xb,\yb}$.
Aldous-Brouder walk on $\Z^d$ and triangular walk are the only RWLMs from \S\ref{rotor walks} for which \eqref{equation: ILC} holds.
The matrix $\Gamma$ is equal to $\frac{1}{d}I_d$ (where $I_d$ is the $d \times d$ identity matrix) in the former case, and is equal to 
$\begin{bmatrix}
	\frac{1}{2} & 0 \\
	0 & \frac{1}{6}
\end{bmatrix}$
in the latter case.
The  
$p$-rotor walk 
does \emph{not} 
satisfy \eqref{equation: ILC} 
as the covariance matrix $\Gamma_{\xb,\yb}$ is equal to 
$\frac{1}{d-1}(I_d-\eb_i\eb_i^\top)$, where $\eb_i$ is the standard unit  vector parallel to the edge $(\xb,\yb)$. 
The  
$p,\!r$-rotor walk (with $r >0$) 
does not  satisfy \eqref{equation: ILC} either by an analogous calculation.

We now restate  Proposition~\ref{prop: scaling limit} from the introduction  in a slightly more general form.

\smallskip

\begin{repproposition}{prop: scaling limit}\label{propscaling limit}
Let $G$ be a simple, connected graph  with its vertex set being a subset of $\R^d$.
	Consider an RWLM  on $G$ that satisfies \eqref{equation: BDD}, \eqref{equation: martingale}, and \eqref{equation: ILC}.
Then the scaled walk $\left\{ \frac{1}{\sqrt{n}} X_{\lfloor nt\rfloor },t \geq 0\right\}$ converges weakly on  $D_{\R^d}[0,\infty)$ to a Brownian motion with diffusion matrix $\Gamma$.
\end{repproposition}

\smallskip

The remarkable part of Proposition~\hyperref[propscaling limit]{1.\thenumberprops} 
is that the conditions involve only  the mechanism of the RWLM,
and hence we can derive a scaling limit result \emph{regardless} of the initial walker-and-rotor configuration.
In particular, it follows from  Proposition~\hyperref[propscaling limit]{1.\thenumberprops} that, for every initial walker-and-rotor configuration,  the triangular walk from Example~\ref{ex: triangular lattice} 
satisfies a functional CLT. 

Naturally,
 Proposition~\hyperref[propscaling limit]{1.\thenumberprops} does not apply to $p$-rotor walk and 
$p,\!r$-rotor walk even when $p=\frac{1}{2}$,
as \eqref{equation: ILC} is never satisfied.
Thus we need a different approach to prove a scaling limit for these RWLMs,
which we partially achieve at the cost of starting the walk 
with a specific  rotor configuration; see Theorem~\ref{theorem: intro scaling limit}. 
 
\smallskip

\begin{proof}[Proof of Proposition~{\hyperref[propscaling limit]{1.\thenumberprops}}]
It suffices to check  that all 
 conditions of Theorem~\ref{CLT} are satisfied. 
 Write $C:= \sup_{ \{\xb,\yb\}\in E} ||\xb-\yb||$.
 Note that $C$ is finite by \eqref{equation: BDD}.
This implies that  $||X_n||\leq Cn+||X_0||$ for all $n\geq 0$,
and it then follows that $(X_n)_{n \geq 0}$ is square-integrable.

 We now check that $(X_n)_{n \geq 0}$ is a martingale process with respect to the filtration $\Fscr_n:=\sigma(X_0,\dots,X_n,\rho_0,\dots,\rho_n)$.
It then follows from the transition rule of RWLM (see \eqref{equation: transition rule RWLM}) that,  for any $n\geq 0$:
\begin{align*}
\Cond{X_{n+1}}{\Fscr_n}
=&\sum_{\xb\in V}\sum_{\yb \in \Nbh(\xb)} \Cond{Y_{\xb,\yb} \, \satubb\{X_n={\xb},\rho_{n}(\xb)=\yb\}}{\Fscr_n}   \qquad (\text{by Definition~\ref{definition: random walk with local memory}})\\
=&\sum_{\xb\in V} \sum_{\yb \in \Nbh{(\xb)}} \E{Y_{\xb,\yb}} \, \satubb\{X_n={\xb},\rho_{n}(\xb)=\yb\}\\
=&\sum_{\xb\in V} \sum_{\yb \in \Nbh(\xb)} \xb \, \satubb\{X_n={\xb},\rho_{n}(\xb)=\yb\} \qquad \text{(by \eqref{equation: martingale})} \\
=&X_n.
\end{align*}
This shows that  $(X_n)_{n \geq 0}$ is a martingale.

We now check the condition \eqref{CLT1}.
It follows from  the the transition rule of RWLM that, for any $n\geq 0$:
\begin{align*}
\Cond{V_n V_n^\top}{\Fscr_n}	&= \sum_{\xb \in V}\sum_{\yb \in \Nbh(\xb)} \Cond{ (Y_{\xb,\yb}- \xb) (Y_{\xb,\yb}- \xb)^\top \,  \satubb \{X_n = \xb, \rho_{n}(\xb)=\yb\}}{\Fscr_n}\\
&= \sum_{\xb \in V}\sum_{ \yb \in \Nbh(\xb)}  \E{ (Y_{\xb,\yb}- \xb) (Y_{\xb,\yb}- \xb)^\top} \,  \satubb\{X_n={\xb},\rho_{n}(\xb)=\yb\}\\
&= \sum_{\xb \in V}\sum_{\yb \in \Nbh(\xb)}  \Gamma \,  \satubb\{X_n={\xb},\rho_{n}(\xb)=\yb \} \qquad \text{(by~\eqref{equation: ILC})}\\
&= \Gamma.
\end{align*}
It then follows that
$\frac{1}{n}\sum_{i=0}^{n-1}\Cond{V_i V_i^\top}{\Fscr_i} = \Gamma$,
which proves \eqref{CLT1}.

We now check the condition \eqref{CLT2}.
Note that 
\[  || V_n||= ||X_{n+1}-X_n||\leq \sup_{ \{\xb,\yb\}\in E} ||\xb-\yb||<\infty,\]
where the last inequality is due to \eqref{equation: BDD}.
Hence for any $\epsilon>0$, for  sufficiently large $n$ we have that $\satubb\{\|V_i\|\geq \epsilon \sqrt{n}\}=0$ for every $i \geq 0$.
This implies that
	\[ \frac{1}{n}\sum_{i=0}^{n-1} \Cond{\|V_i\|^2 \satubb\{\|V_i\| \geq \epsilon \sqrt{n}\}}{\Fscr_i}= 0,\]
which proves \eqref{CLT2}.
The proof is now complete.	
\end{proof}

\bigskip

\section{Wired spanning forest oriented toward a root}\label{ust}

In this section  we present two methods to generate the wired spanning forest   oriented toward a chosen root vertex, which we will use to construct  an initial rotor configuration for random walks with local memory in \S\ref{stationarity} and \S\ref{SLLN}.
Most of the material in this section is not new.
Indeed, the
 material in  \S\ref{subsection: unoriented wired spanning forest} and \S\ref{subsection: wsf oriented toward a root}
 is taken from the relevant part of \cite{LP16},
and the  material in \S\ref{subsection: Wilson recurrent} and \S\ref{subsection: Wilson transient} is 
 a straightforward modification of Wilson's method~\cite{Wil96,BLPS01}, which we spell out for completeness.

 \subsection{Unoriented wired spanning forest}\label{subsection: unoriented wired spanning forest}
We begin by defining the unoriented wired spanning forest, and we
 refer to  \cite{BLPS01} and    
 \cite[Chapters 4 \& 10]{LP16} for a  detailed discussion on this topic.

Recall that  $G:=(V(G),E(G))$ is a simple, connected, 
 undirected graph   that  is locally finite. Let  $\Fscr:=\Fscr(G)$ be the $\sigma$-algebra on the set of  subgraphs of $G$  generated by sets of the form
$\{H \in 2^{E(G)}  \mid B \subseteq H\}$, where  $B$ is a finite subset of $E(G)$. The unoriented wired spanning forest will be a probability distribution on the measurable space $(2^{E(G)}, \Fscr(G))$.


An \emph{electrical network} is a pair $(G,c)$, where $G$ is a locally finite, simple, connected  graph, and 
the \emph{conductance} $c:E \to \R_{>0}$  is a function that sends each unoriented edge of $G$ to a positive real number. We denote by $c\{x,y\}$ the conductance of the unoriented edge $\{x,y\}$.
 (We emphasize that $G$ is always an \textbf{unoriented} graph, and $c\{x,y\} = c\{y,x\}$.)
 
We associate to each 
 $(G,c)$  the Markov chain with state space $V(G)$ and such that,
 for every adjacent vertices $x,y$,  
  the probability to transition from $x$ to $y$  is proportional to $c\{x,y\}$. 
This Markov chain is called the \emph{network random walk} on $(G,c)$.
The network $(G,c)$ is  \emph{recurrent} if the network random walk eventually returns to its starting point with probability $1$, and is \emph{transient} otherwise. 

We start by defining the wired spanning forest for the network $(G,c)$ when $G$ is a finite graph, in which case the distribution is concentrated on the spanning trees of $G$.
The \emph{weight} of a finite 
subgraph $H$ of $G$ is
\[\Xi(H):=\prod_{\{x,y\} \in E(H)} c\{x,y\}. \]


\smallskip

\begin{definition}[Unoriented spanning forest for finite graphs]\label{definition: unoriented wsf finite graph}
For a finite graph $G$,
the \emph{unoriented wired spanning forest} $\wsf:=\wsf(G,c)$ is the probability distribution on spanning trees of $G$ in which each tree $T$ is picked with probability proportional to $\Xi(T)$.
\end{definition}

\smallskip

 Note that the term ``wired'' is not usually present in Definition~\ref{definition: unoriented wsf finite graph}  when $G$ is finite, as wired exhaustion (see Definition~\ref{definition: wired exhaustion} below) is not a required concept here.
 In fact, in this case  the wired spanning forest will  always be a tree. 
However, using the terms ``wired'' and ``forest'' will significantly simplify the notation in this paper, as  our results  apply to both finite and infinite graphs.

\smallskip 
 
\begin{definition}[Wired exhaustion]\label{definition: wired exhaustion}
Let $(W_n)_{n \geq 0}$ be a sequence of finite, connected subsets of $V(G)$ such that 
\begin{itemize}
 \item $\bigcup_{n \geq 0}W_n=V(G)$; and 
 \item $W_n \subseteq W_{n+1}$ for all $n \geq 0$.
\end{itemize}
 The {\em wired exhaustion} of $G$ is the sequence of electrical networks $(G_n, c_n)_{n \geq 0}$  defined as follows.
 The graph  $G_n$ is the undirected graph    obtained from $G$ by identifying all the vertices of $V(G)\setminus W_n$ to a single vertex $z_n$ and removing  loops and extra multiple edges that are formed.
 The conductance  $c_n: E(G_n) \to \R_{>0}$ is defined  by 
 \[ c_n\{x,y\}:= \begin{cases}
 c\{x,y\}& \text{ if } x,y \in W_n;\\
 \sum_{y' \notin W_n} c\{x,y'\} & \text{ if } x \in W_n \text{ and } y=z_n.  
\end{cases} \qedhere
   \]
\end{definition}

\smallskip

We denote by $\mu_n$ the probability distribution $\wsf(G_n,c_n)$ on the subgraphs of $G_n$.
We can now define  the wired spanning forest for infinite graphs using the concept of  wired exhaustion.  

\smallskip

\begin{definition}[Unoriented wired spanning forest for infinite graphs]\label{definition: unoriented wsf for all graphs}
The \emph{wired spanning forest} ${\wsf}:={\wsf}(G,c)$ is a probability distribution on subgraphs of $G$ such that, for any wired exhaustion and any finite $B \subseteq E(G)$,
\begin{equation}\label{equation: infinite-volume limit unoriented wsf}
	\wsf[B\subseteq F] = \lim_{n \to \infty} \mu_n[B \subseteq T_n],
\end{equation}
where $F$ is a random subgraph of $G$ distributed according to $\wsf$, and $T_n$ is a random spanning tree of $G_n$ distributed according to $\mu_n$.
\end{definition}

\smallskip

The quantity $\mu_n[B \subseteq T_n]$  decreases as $n \to \infty$~\cite[Chapter~10]{LP16}, and hence 
the limit in \eqref{equation: infinite-volume limit unoriented wsf} exists  and does not depend on the choice of the wired exhaustion.
By the Kolmogorov extension theorem, there exists a unique probability distribution on $(2^{E(G)}, \Fscr(G))$ that satisfies \eqref{equation: infinite-volume limit unoriented wsf}.   

The random subgraph sampled from ${\wsf}$ is always a spanning forest but not necessarily a  spanning tree.
It is  well-known that, for the graph $\Z^d$ with a constant conductance, this random subgraph has  one  connected component a.s.\ if $d\leq 4$, and  infinitely many  connected components a.s.\ if $d\geq 5$~\cite[Theorem~4.2]{Pem91}. For more on the geometry of the ${\wsf}$ and its dependence on dimension, see \cite{BKPS04,HP19}.

\medskip

\subsection{Wired spanning forest oriented toward a root}\label{subsection: wsf oriented toward a root}
 We now define the  wired spanning forest oriented toward a chosen root vertex. 
 Denote by 
 \[\arrow{E}(G) := \bigcup_{\{x,y\} \in E(G)} \{ (x,y), (y,x) \}\]
 the set of oriented edges of $G$. Let  $\Fscrl:=\Fscrl(G)$ be the $\sigma$-algebra on the set of oriented subgraphs of $G$  generated by sets of the form
$\{\Hl \in 2^{\El(G)}  \mid \Bl \subseteq \Hl\}$, where  $\Bl$ is a finite subset of $\El(G)$.
The oriented wired spanning forest will be a probability distribution on the measurable space $(2^{\El(G)}, \Fscrl(G))$.

We start by defining the oriented wired spanning forest when $G$ is a finite graph, in which case the distribution is concentrated on the oriented spanning trees of $G$.
Fix a root vertex $r \in V(G)$ for the rest of this section.

\smallskip

\begin{definition}[Oriented spanning tree]\label{oriented spanning tree}
An \emph{$r$-oriented spanning tree} $\Tl$ of $G$ is an oriented subgraph of $G$ such that,  for any $x \in V(G)$, there exists a unique  directed path in $\Tl$ that starts at $x$ and ends at $r$.
\end{definition}

\smallskip

Note that in an $r$-oriented spanning tree $\Tl$, every vertex in $V(G) \setminus\{r\}$ has outdegree $1$ in $\Tl$, and the root vertex $r$ has outdegree $0$ in $\Tl$.
Also note that given an unoriented spanning tree of a finite graph and a root vertex $r$, there is a unique way to orient the tree to become an $r$-oriented spanning tree.
The \emph{weight} of a finite oriented subgraph $\arrow{H}$ of $G$  is
\[ \Xi(\arrow{H}):=\prod_{(x,y) \in \El(\arrow{H})} c\{x,y\}.  \qedhere \]

\smallskip

\begin{definition}[Rooted oriented wired spanning forest for finite graphs]
\label{definition: wsf finite graphs}
Let $G$ be a finite graph.
The $r$-\emph{oriented wired spanning forest}, denoted $\arrow{\wsf}_r:=\arrow{\wsf}_r(G,c)$, 
is the probability distribution on $r$-oriented spanning trees of $G$ in which each tree $\Tl$ is picked with probability proportional to $\Xi(\Tl)$.
\end{definition}

 \smallskip

We  now define  the $r$-oriented wired spanning forest for infinite graphs $G$. Let $(G_n,c_n)_{n \geq 0}$ be a wired exhaustion of $G$. 
We denote by $\arrow{\mu_{r,n}}$ the probability distribution 
$\arrow{\wsf}_r(G_n,c_n)$ on the oriented subgraphs of $G_n$.

\smallskip

\begin{definition}[Rooted oriented wired spanning forest for infinite graphs]\label{definition: wsf all graphs}
The \emph{$r$-oriented wired spanning forest}, denoted $\arrow{\wsf}_r:=\arrow{\wsf}_r(G,c)$, 
is a probability distribution on oriented subgraphs of $G$ 
such that, for any wired exhaustion
and any  
finite   $\Bl \subseteq \El(G)$,
\begin{equation}\label{equation: infinite-volume limit oriented wsf}
 \arrow{\wsf}_r[\Bl\subseteq \arrow{F}] =\lim_{n \to 0} \  \arrow{\mu_{r,n}}[\Bl \subseteq \arrow{T_n}],  
 \end{equation}
where $\arrow{F}$ is a random oriented subgraph of $G$ distributed according to $\arrow{\wsf}_r$ and $\arrow{T_n}$ is a random $r$-oriented tree of $G_n$ distributed according to $\arrow{\mu_{r,n}}$.
\end{definition}

\smallskip

The limit in Definition~\ref{definition: wsf all graphs} exists  and does not depend on the choice of the wired exhaustion
 as we will see in \S\ref{subsection: Wilson recurrent} (for recurrent networks) and  \S\ref{subsection: Wilson transient} (for transient networks).
By the Kolmogorov extension theorem, there exists a unique probability distribution on $(2^{\El(G)}, \Fscrl(G))$ that satisfies Definition~\ref{definition: wsf all graphs}.

The  underlying graph of the $r$-oriented wired spanning forest is the unoriented wired spanning forest, in the following sense.

\smallskip

\begin{lemma}\label{lemma: wsf and oriented wsf}
Let $f: 2^{\arrow{E}(G)} \to 2^{E(G)}$ be the map that takes an oriented subgraph and  erases the orientation of every edge. 
If $\Fl$ is an oriented subgraph of $G$ sampled from $\arrow{\wsf}_r$, 
then $f(\Fl)$ is an unoriented subgraph of $G$ that has the law of $\wsf$.
\end{lemma}

\smallskip

\begin{proof}
Note that, 
for any finite subset $B$ of $E(G)$,
the event $\{ B \subseteq f(\Fl)\}$ depends only on finitely many oriented edges.
Therefore, it suffices to consider the case when $G$ is a finite graph, as the case of infinite graphs follows by taking the limit over a wired exhaustion and then verifying the lemma 
for all events of the form $\{ B \subseteq f(\Fl)\}$ for some finite $B$.

When $G$ is finite, note that 
 $f$ is a bijection between 
$r$-oriented spanning trees of $G$ and  unoriented spanning trees of $G$
that preserves the weight of spanning trees.
The lemma now 
  follows from  Definition~\ref{definition: unoriented wsf finite graph} and Definition~\ref{definition: wsf finite graphs}, and the proof is complete.
\end{proof}

\smallskip

As in the unoriented case, a random oriented  subgraph $\arrow{F}$ sampled from $\arrow{\wsf}_r$ is not necessarily an oriented  spanning tree.
However, it is always an \emph{$r$-oriented spanning forest} of $G$:
the underlying graph of $\arrow{F}$ is a spanning forest,
 every vertex in $V(G)\setminus \{r\}$ has outdegree 1 in $\arrow{F}$, and $r$ has outdegree $0$ in $\arrow{F}$.
The first condition follows from Lemma~\ref{lemma: wsf and oriented wsf}, and the others can be verified directly from the limit in Definition~\ref{definition: wsf all graphs} as these events only depend on  finitely many edges.

\medskip

\subsection{Wilson's method oriented toward a root: recurrent case}\label{subsection: Wilson recurrent}

In this subsection we describe an algorithm due to Wilson~\cite{Wil96} that  generates $\wsf(G,c)$ and $\arrow{\wsf}_r(G,c)$ for recurrent networks without using the weak limit in Definition~\ref{definition: wsf all graphs}.

A \emph{(finite) directed walk}  in $G$ is a sequence $\langle x_0,  \ldots, x_n \rangle$ such that $\{x_i,x_{i+1}\} \in E(G)$ for $i \in \{0,\ldots, n-1\}$.
The  \emph{loop erasure} of a directed walk $\langle x_0,  \ldots, x_n \rangle$, denoted by $\LE\langle x_0, \ldots, x_n \rangle$,  is  obtained by erasing cycles in the directed walk  in the order they appear,
i.e., it is the directed walk given by the following recursive definition.
Let $y_0:=x_0$.
Suppose that $y_i$ has been defined, and let $j$ be the largest element of $\{0,\ldots,n\}$ such that $x_j=y_i$.
Set $y_{i+1}:=x_{j+1}$ if $j <n$; otherwise, define $\LE\langle x_0,  \ldots, x_n \rangle:=\langle y_0, \ldots y_i \rangle$. 
Note that even if the directed walk is infinite, 
its loop erasure is still 
 well-defined provided that the walk is \emph{locally finite}, i.e., every vertex is visited at most finitely many times in the walk.

\smallskip

\begin{definition}
[Wilson's method for recurrent networks]
Let $(G,c)$ be a recurrent network.
Let $x_1,x_2,\ldots $ be an ordering of elements of the $V(G) \setminus \{r\}$.
Define a growing sequence 
$(\Tl(i))_{i \geq 0}$ of oriented trees recursively as follows:
\begin{itemize}
\item Set 
$\Tl(0)$ 
to be the tree with  the single vertex $r$ and with no edges.
\item Suppose that 
$\Tl(i)$ 
has been generated.
Start an independent network random walk at $x_{i+1}$ and stop it at the first time it hits 
$\Tl(i)$ 
(note that the random walk hits $\Tl(i)$ a.s.\ by recurrence).
Let $\langle y_0,\ldots, y_m \rangle$ be the loop erasure of this random walk.
\item 
Set 
$\Tl(i+1)$
to be the oriented tree obtained by adding the
oriented
 edges 
$(y_0,y_1)$, $(y_1,y_2)$, $\ldots$, $(y_{m-1},y_m)$ to $\Tl(i)$.
\item The output of this algorithm is  
 $\Tl:=\bigcup_{i \geq 0} \Tl(i)$. 
\end{itemize}
\end{definition}

\smallskip

The oriented spanning forest sampled using Wilson's method has the law of the $r$-oriented wired spanning forest, due to the following  theorems.

\smallskip

\begin{theorem}[{\cite[Theorem 1]{Wil96}}]\label{theorem: wsf finite}
Let  $G$ be a finite graph.
Then, regardless of the  ordering of $V(G) \setminus \{r\}$, the oriented tree $\Tl$ sampled using Wilson's method  has the law of
$\arrow{\wsf}_r(G,c)$. \qed 
\end{theorem}

\smallskip

\begin{theorem}[{\cite[Proposition~5.6]{BLPS01}}]\label{theorem: wsf recurrent}
Let $(G,c)$ be a recurrent network.
Then for any finite subset $\Bl$ of $\arrow{E}(G)$, any ordering of $V(G) \setminus \{r\}$, and any wired exhaustion of $G$,
\[ \mathbb{P}[\Bl\subseteq \Tl] =\lim_{n \to 0} \  \arrow{\mu_{r,n}}[\Bl \subseteq \Tl_n],  \]
where $\Tl$ is a random  tree of $G$ generated using Wilson's method, 
and ${\Tl_n}$ is a random  tree of $G_n$ distributed according to $\arrow{\mu_{r,n}}$.
\qed 
\end{theorem}

\smallskip

We remark that 
\cite{BLPS01} stated only the unoriented version of Theorem~\ref{theorem: wsf recurrent}, but their argument in fact proves the oriented version as well.
As a consequence of Theorem~\ref{theorem: wsf recurrent}, we have that, for every recurrent network, the limit in \eqref{equation: infinite-volume limit oriented wsf}
exists and does not depend on the choice of the wired exhaustion.

\medskip

\subsection{Wilson's method oriented toward a root: transient case}\label{subsection: Wilson transient}

 In this subsection we describe an algorithm  that generates $\arrow{\wsf}_r(G,c)$ for transient networks without using the weak limit in Definition~\ref{definition: wsf all graphs}.

For any walk $\langle x_i \mid 0 \leq i < I \rangle$ (including the case $I=\infty$), we denote by  $\El(\langle x_i \mid 0 \leq i < I \rangle)$ the set of oriented edges $\{(x_i,x_{i+1}) \mid 0 \leq i < I-1\}$, and we denote by $\El(\mathsf{R}(\langle x_i \mid 0 \leq i < I\rangle))$ the set of oriented edges $\{(x_{i+1},x_i) \mid 0 \leq i < I-1\}$.

\smallskip

\begin{definition}[Wilson's method for transient networks]\label{definition: Wilson method transient network}
Let $(G,c)$ be a transient network.
Let $x_1,x_2,\ldots $ be an ordering of elements of $V(G) \setminus \{r\}$.
Define a growing sequence $(\Fl(i))_{i \geq 0}$ of oriented forests recursively as follows:

\begin{itemize}
\item Start a network random walk at $r$ that runs indefinitely. 
This random walk is locally finite a.s.\ by transience. 
Let $\langle y_0,y_1,\ldots\rangle$ be the loop erasure of this random walk.
Set $\Fl(0)$ to be the tree oriented toward $r$ given by
\[ V(\Fl(0)):=\{y_i \mid i \geq 0 \}; \qquad \El(\Fl(0)):=\El(\mathsf{R}(\langle y_i  \mid i \geq 0\rangle)).  \]

\item Suppose that $\Fl(i)$ has been generated.
Start a network random walk at $x_{i+1}$. 
Stop the walk the first time it hits $\Fl(i)$; if it never hits $\Fl(i)$ then let it run indefinitely.
This walk is locally finite a.s.\ by transience.
Let $\langle y_0',y_1',\ldots\rangle$ be the loop erasure of this random walk.

\item Set $\Fl(i+1)$ to be the oriented forest obtained by adding the edges in $\El(\langle y'_i  \mid i \geq 0\rangle)$ to $\Fl(i)$.

\item The output of this algorithm  is  $\Fl:=\bigcup_{i \geq 0} \Fl(i)$.  \qedhere
\end{itemize}
\end{definition}
 
\smallskip 
 
We remark that this algorithm is identical to Wilson's method oriented toward infinity~\cite{BLPS01} except for the first step, where we take  the oriented edges from $\El(\mathsf{R}(\langle y_i  \mid i \geq 0\rangle))$ instead of $\El(\langle y_i  \mid i \geq 0\rangle)$.
This difference  causes the output to be a forest oriented toward $r$ instead of  toward infinity.
We refer to \cite{BLPS01} and \cite{Hut18} for other methods  to sample wired spanning forest oriented toward infinity.

The subgraph  sampled using this method has the law of the $r$-oriented wired spanning forest, due to the following  theorem.

\smallskip

\begin{theorem}[cf.{\cite[Theorem~5.1]{BLPS01}}]\label{theorem: wsf transient}
Let $(G,c)$ be a transient network.
Then for any finite subset $\Bl$ of $\El(G)$, any  ordering of $V(G) \setminus \{r\}$, and any wired exhaustion of $G$,
\[ \mathbb{P}[\Bl\subseteq \Fl] =\lim_{n \to 0} \  \arrow{\mu_{r,n}}[\Bl \subseteq \arrow{T_n}],  \]
where $\Fl$ is a random oriented forest of $G$ generated using Wilson's method oriented toward $r$ (Definition~\ref{definition: Wilson method transient network}), and $\arrow{T_n}$ is a random oriented tree of $G_n$ distributed according to $\arrow{\mu_{r,n}}$.
\end{theorem}

\smallskip

As a consequence of Theorem~\ref{theorem: wsf transient}, we have that for all transient networks the limit in \eqref{equation: infinite-volume limit oriented wsf} 
exists and does not depend on the choice of the wired exhaustion.
 Our proof of Theorem~\ref{theorem: wsf transient} is paraphrased from its counterpart in \cite{BLPS01}.
 
 \smallskip
 
\begin{proof}[Proof of Theorem~\ref{theorem: wsf transient}]
For any locally finite walk $\langle x_i  \mid i \geq 0\rangle$, 
we have $\LE\langle x_i  \mid i < I\rangle \to \LE\langle x_i  \mid i \geq 0\rangle$ as $I \to \infty$.
That is, if $\LE\langle x_i  \mid i \leq I\rangle=\langle y_{I,i} \mid i \leq m_I\rangle$ and $\LE\langle x_i  \mid i \geq 0 \rangle=\langle y_{i} \mid i \geq 0\rangle$, then for every  $i$  and all sufficiently large $I$ we have $y_{I,i}=y_i$.
Since $G$ is transient,
it follows that $\LE\langle X_i  \mid i < I\rangle \to \LE\langle X_i  \mid i \geq 0\rangle$ as $I \to \infty$ a.s., where 
$\langle X_i  \mid i \geq 0\rangle$ is a network random walk starting from any fixed vertex of $G$.

Let $x_1,x_2,\ldots$ be the ordering of $V(G) \setminus \{r\}$ used in  Wilson's method for $G$.
Write $x_0:=r$.
Let $L$ be a sufficiently large integer such that the endpoints of all edges in $\Bl$ are contained in $x_0,x_1,\ldots, x_L$.
Let $\langle X_i^j \mid i \geq 0\rangle$ be independent random walks on $G$  that start at $x_j$ ($j \in \{0,\ldots,L\}$).

Let $n$ be sufficiently large so that the wired exhaustion $W_n$ contains 
$x_0,\ldots, x_L$.
Run Wilson's method rooted at $z_n$ in $G_n$ with an ordering of $V(G_n)\setminus \{z_n\}$ that starts with $x_0,\ldots,x_L$, using the walks $\langle X_i^j \mid i \geq 0 \rangle$ for $j \in \{0,\ldots,L\}$.
Since these walks are on $G$ rather than $G_n$, we simply stop the random walks once they leave the set $W_n$ and say that they have hit $z_n$.
In this way, we can couple the random walk in $G_n$ that starts at $x_j$ with the random walk in $G$ that starts at $x_j$
 by using the same (infinite) random walk $\langle X_i^j \mid i \geq 0\rangle$ for $j \in \{0,\ldots,L\}$.

Let $\Tl'_n$ be the random spanning tree of $G_n$ oriented toward $z_n$ picked using  Wilson's method for $G_n$ as  described in the previous paragraph.
Note that $\Tl'_n$ has the law of $\arrow{\wsf}_{z_n}(G_n,c_n)$ by  Theorem~\ref{theorem: wsf finite}.

Let $h$ be the map from $z_n$-oriented spanning trees of $G_n$ to $r$-oriented spanning trees of $G_n$ that reverses the orientation of all edges in the unique directed path from $r$ to $z_n$.
Note that $h$ is a bijection that preserves the weight of spanning trees.
Write $\Tl_n:=h(\Tl_n')$.
It then follows from definition of oriented wired spanning forest for finite graphs (Definition~\ref{definition: wsf finite graphs}) that 
$\Tl_n$ has the law of $\arrow{\wsf}_{r}(G_n,c_n)$.

Let $\tau^j_n$ be the first time that $\langle X_i^j \mid i \geq 0\rangle$ reaches the portion of the spanning tree created by the preceding random walks $\langle X_i^l \mid i \geq 0\rangle$ for $(l <j)$ using Wilson's method for $G_n$ oriented toward $z_n$. 
Note that we have:
\begin{equation}\label{equation: wilson transient 1}
  \arrow{\mu_{r,n}}[\Bl \subseteq \Tl_n]=  \mathbb{P} \left[ \Bl \subseteq  \, \El(\mathsf{R}(\LE\langle X_i^0 \mid i \leq \tau_n^0 \rangle))  \, \cup \, \bigcup_{j=1}^L \El(\LE\langle X_i^j \mid i \leq \tau_n^j \rangle)  \right].
  \end{equation}

Let $\tau^j$  be the 
first time that $\langle X_i^j \mid i \geq 0\rangle$ reaches the portion of the spanning tree created by the preceding random walks $\langle X_i^l \mid i \geq 0\rangle$ for $(l <j)$ using Wilson's method for $G$ oriented toward $r$. 
Note that we have
\begin{equation}\label{equation: wilson transient 2}
 \mathbb{P} [\Bl \subseteq \Fl]=  \mathbb{P} \left[ \Bl \subseteq  \, \El(\mathsf{R}(\LE\langle X_i^0 \mid i \leq \tau^0 \rangle))  \, \cup\,  \bigcup_{j=1}^L \El(\LE\langle X_i^j \mid i \leq \tau^j \rangle) \right],
  \end{equation}
  where $\Fl$ is the  oriented spanning forest generated using Wilson's method for $G$.
Since the random walks used in Wilson's method for $G$ and Wilson's method for $G_n$ are the same, it follows from induction on $j$ that $\tau_n^j\to \tau^j$ as $n \to \infty$.
Together with \eqref{equation: wilson transient 1} and \eqref{equation: wilson transient 2}, this implies the conclusion of the theorem.
\end{proof}

\medskip

\subsection{Tail triviality}\label{subsection: tail triviality}
An important property of the wired spanning forest (which will be used in proving Theorem~\ref{theorem: intro ergodic})
is that it is a tail trivial measure.


We first define tail triviality for measures on unoriented subgraphs.  
For any subset $K \subseteq E(G)$, let $\Fscr(K) \subseteq \Fscr$ denote the $\sigma$-algebra of events that depend only on $K$.
An event  $\Bscr \in \Fscr$   is a \emph{tail event} if $\Bscr \in \Fscr(E \setminus K)$ for all finite $K \subseteq E$.
A measure $\pi$ on $\Fscr$ is \emph{tail trivial} if, for every tail event $\Bscr \in \Fscr$, we have $\pi[\Bscr] \in \{0,1\}$. 

\smallskip

\begin{theorem}[{\cite[Theorem 10.18]{LP16}}]\label{lemma: wsf is tail trivial}
	For every tail event $\Bscr \in \Fscr$, we have $\wsf[\Bscr] \in \{0,1\}$. \qed
\end{theorem}

\smallskip

We now define tail triviality for measures on oriented subgraphs analogously.  
For any subset $\Kl \subseteq \El(G)$, let $\Fscrl(\Kl) \subseteq \Fscrl$ denote the $\sigma$-algebra of events that depend only on $\Kl$.
	An event  $\Bscrl \in \Fscrl$   is a \emph{tail event} if $\Bscrl \in \Fscrl(\El \setminus \Kl)$ for all finite $\Kl \subseteq \El$. A measure $\pi$ on $\Fscrl$ is \emph{tail trivial} if, for every tail event $\Bscrl \in \Fscrl$, we have $\pi[\Bscrl] \in \{0,1\}$. 
We now show that the following oriented subgraph measure is tail trivial.

\smallskip 

\begin{definition}[Oriented wired spanning forest plus one edge]\label{definition: oriented wired spanning forest plus one edge}
	The \emph{$r$-oriented wired spanning forest plus one edge}, denoted  $\arrow{\wsf}_r^+:=\arrow{\wsf}_r^+(G,c)$,  is the law of the random subgraph $\Fl \sqcup \{(r, Y)\}$, where $\Fl$ is a random $r$-oriented forest of $G$ sampled from   $\arrow{\wsf}_r$ and $Y$ is a random neighbor of $r$ sampled from $\mu_r$ independently of $\arrow{F}$.
\end{definition}

\smallskip

%


\begin{lemma}\label{lemma: wsfp is tail trivial}
	For every tail event $\Bscrl \in \Fscrl$, we have $\arrow{\wsf}^+_r[\Bscrl]\in \{0,1\}$.
\end{lemma}
\begin{proof}
	Let $f: 2^{\El} \to 2^{E}$ be the map that takes an oriented subgraph and erases the orientation of every edge.
	Let $g: 2^{\El} \to 2^{\El}$ be the map that takes an oriented subgraph and removes any outgoing edges of $r$.

	Let $\Bscrl$ be a tail event in $\Fscrl$.
	Note that 
	$\arrow{\wsf}^+_r[\Bscrl]=\arrow{\wsf}_r[g(\Bscrl)]$ by 
	the definition of $\arrow{\wsf}^+_r$ and by the fact that $\Bscrl$ 
	does not depend on any outgoing edges of $r$.
	Also note that  $\arrow{\wsf}_r[g(\Bscrl)]=\wsf[f\circ g(\Bscrl)]$ by 
	Lemma~\ref{lemma: wsf and oriented wsf}. 
	Finally, note that 
	the set $f \circ g(\Bscrl)$ is a tail event in $\Fscr$ since $\Bscrl$ is a tail event in $\Fscrl$.
	The conclusion of the lemma now follows from the tail triviality of unoriented wired spanning forest (Theorem~\ref{lemma: wsf is tail trivial}).
\end{proof}

\bigskip

\section{A native environment for random walk with local memory}\label{stationarity}

In this section we show that the wired spanning forest measure  can be used to construct a native environment.
To rigorously define the notion of native environment,
the underlying RWLM needs to satisfy the conditions described below.


%

A graph $G$ is a \emph{Cayley graph} if 
\begin{itemize}
\item $V(G)$ is a group with identity element $\id$;
\item The group $V(G)$ is generated by a finite set $\Sc \subseteq V(G) \setminus \{\id\}$;
\item The set $\Sc$ is symmetric, i.e., if $x$ is in $\Sc$ then $x^{-1}$ is also in $\Sc$; and 
\item  $E(G)=\{ \{x,y\} \mid y^{-1}x \in \Sc  \}$. 
\end{itemize} 
The square lattice $\Z^2$ is an example of a Cayley graph
where the generating  set $\Sc$ is  $\{(\pm 1,0),(0,\pm 1)\}$ and  the group operation is vector addition.
Note that a Cayley graph  is locally finite (because  $\Sc$ is finite),  
 connected (because $\Sc$ is a generating set), and simple (because $\Sc$ does not contain $\id$).

A \emph{weighted Cayley graph}  $(G, c)$  is  a Cayley graph $G$
with a weight function  $c: \Sc \to \R_{>0}$ such that $c(x)=c(x^{-1})$ for all $x \in \Sc$. 
Note that the function $c :\Sc\to \R_{>0}$ extends naturally to a conductance  $c: E\to \R_{>0}$ on edges of $G$ by setting $c\{x,y\}:=c(y^{-1}x)= c(x^{-1}y)$ for all $\{x,y\} \in E$.

Recall the definition of the probability transition function $p_x(\cdot, \cdot)$ from Definition~\ref{definition: mechanism}.
For every vertex $x$ of $G$,
we denote by $\mu_x$ the probability distribution on neighbors of $x$ given by 
\begin{equation}\label{equation: mu}
	\mu_{x}(y) \ := \ \frac{c\{x,y\}}{\sum_{z \in \Nbh(x)} c\{x,z\}} \qquad (y \in \Nbh(x)).
\end{equation}

Note that the measure $\mu_{\id}$ is symmetric (i.e., $\mu_{\id}(x)=\mu_{\id}(x^{-1})$) as a consequence of  $c: \Sc \to \R_{>0}$ being symmetric.

An RWLM is \emph{transitive} if,
\begin{equation}\label{equation: transitive}
	p_x(y,y') = p_{gx}(gy,gy') \qquad \text{for every } \  x, g,y,y' \in V.  \tag{Tran}
\end{equation}
An RWLM  is \emph{c-stationary}  if, for every vertex $x$,
\begin{equation}\label{equation: c-stationary}
	\text{$\mu_x$ is a stationary distribution of the local chain $M_x$.}
	\tag{cSta}
\end{equation}
Intuitively, the transitivity condition requires that  the RWLM's mechanism at every vertex follow the same procedure.
We remark that  every  RWLM in \S\ref{rotor walks}, with $c$ being a constant function, is transitive and  $c$-stationary.

%

For the rest of this paper, every RWLM will be transitive and $c$-stationary, and the underlying graph will always be a weighted Cayley graph, unless stated otherwise.
Recall that $X_n$ denotes the location of the walker and $\rho_n$ denotes the rotor configuration at the $n$-th step of  RWLM.

\smallskip

\begin{definition}[Scenery process]\label{definition: scenery process}
The    \emph{scenery process}  is the sequence $(\rel_n)_{n \geq 0}$ of rotor configurations given by
	\[\rel_n(x):=X_n^{-1}  \rho_n ( X_n x)  \qquad {(x\in V,\, n \geq 0).} \]
\end{definition}

\smallskip

Described in words, at each time step we apply a translation to the current rotor configuration so that the current location of the walker is mapped to the origin.
In this way,  $\rel_n$ is the rotor configuration \emph{as viewed from the perspective of the walker} at the $n$-th step of the RWLM.
		See Figure~\ref{figure: scenery process} for an illustration of a  scenery process.

Note that, as a consequence of \eqref{equation: transitive}, 
the scenery process $(\rel_n)_{n \geq 0}$ is a Markov chain with state space the set of rotor configurations of $G$ and with transition rule
\begin{equation}\label{equation: transition rule scenery process}
 \rel_{n+1}(x) :=
\begin{cases}\id &\text{ if } x= Y_{n}^{-1};\\
Y_n^{-1}\rel_n( Y_{n} x) &\text{ if } x\neq Y_{n}^{-1},
\end{cases}
\end{equation}
where $Y_n$ is a random neighbor of $\id$ sampled from $p_{\id}(\rel_n(\id),\cdot)$ independently of $\sigma(\rel_0,\ldots, \rel_{n-1})$ (recall that $p_{\id}$ is the probability transition function of the local chain $M_{\id}$).

	\begin{figure}[t!]
\includegraphics[scale=0.7]{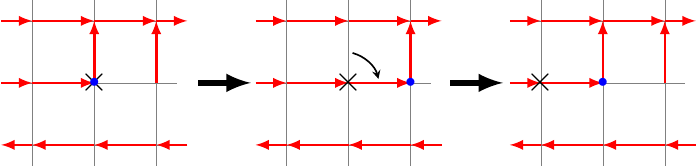} 
\caption{One step of the scenery process of a rotor walk on $\Z^2$ with clockwise rotation as its mechanism.
The location of the origin in the original process is marked by  the {$\times$} symbol, and the location of the walker is marked by the {\color{blue}$\bullet$} symbol.}
\label{figure: scenery process}
\end{figure}

\smallskip

\begin{definition}[Native environment]\label{definition: native environment}
A \emph{native environment}  is a probability distribution on   rotor configurations of $G$ 
such that, if the walker starts at $\id$ and 
the initial rotor configuration is sampled from the distribution,
then the scenery process 
  is a stationary sequence, i.e.,
   \[ (\rel_n)_{n \geq 0} \ \overset{d}{=} \ (\rel_{n+1})_{n \geq 0}. \] 
\end{definition}

\smallskip

Intuitively, a native environment means that, at each time step of the walk,
the rotor configuration  \emph{viewed from the perspective of the walker}  has the same law as the initial environment.
See Figure~\ref{figure: scenery process} for an illustration of a  native environment.

%
%

We now restate Theorem~\ref{theorem: intro stationarity} from the introduction (also the main result of this section) in a slightly more general form.
Recall the definition of $\arrow{\wsf}^+_{\id}:=\arrow{\wsf}^+_{\id}(G,c)$ from  Definition~\ref{definition: oriented wired spanning forest plus one edge}.

\smallskip

\begin{reptheorem}{theorem: intro stationarity}\label{thmstationarity}
	Consider an RWLM  on a weighted Cayley graph  that satisfies \eqref{equation: transitive} and \eqref{equation: c-stationary}.
 Then  $\arrow{\wsf}^+_{\id}$ is a native environment.
\end{reptheorem}

\smallskip

Note that  
$\arrow{\wsf}^+_{\id}$ 
is indeed a probability distribution  on rotor configurations of $G$.
This is because, by Wilson's method (see \S\ref{subsection: Wilson recurrent} and \S\ref{subsection: Wilson transient}),  the random subgraph sampled from  $\arrow{\wsf}_{\id}$ has exactly one outgoing edge for every $x \in V \setminus \{\id\}$ and no outgoing edge for $\id$. 
Hence  the random subgraph  sampled from $\arrow{\wsf}_{\id}^+$
has exactly one outgoing edge for every vertex, and  by  Remark~\ref{remark: rotor configurations dual}
it defines  a rotor configuration of $G$.

We remark that, when $G$ is a finite Cayley graph, 
Theorem~\hyperref[thmstationarity]{1.\thenumberstationarity} then specializes to the result of \cite{Bro89,Ald90} (for Aldous-Broder walk) and 
\cite[Lemma~3.4]{HLM08} (for rotor walk).

We now build toward the proof of Theorem~\hyperref[thmstationarity]{1.\thenumberstationarity}.
We will use the following identity, which 
 is a special case of {\cite[Lemma~2.4]{Lev11}} if the graph $G$ is finite.

\smallskip

\begin{lemma}\label{lemma: wsfplus  sum}
Let $(G,c)$ be an electrical network, and let $r$ be a vertex.
Let $Y$ be a random neighbor of $r$ sampled from $\mu_r$, and  let $\arrow{F_{Y}}$ be a random oriented spanning forest of $G$ sampled from  $\arrow{\wsf}_{Y}$.
Then the random oriented subgraph $\arrow{F_{Y}}  \, \sqcup \, \{ (Y, r) \}$ has the distribution $\arrow{\wsf}^+_r$.
\end{lemma}

\smallskip

\begin{proof}
It suffices to consider the case when $G$ is a finite graph,
as  the case of infinite graphs follows  by  taking the limit over a wired exhaustion and then  verifying the lemma for all events that depend on only finitely many edges.

When $G$ is a finite graph, note that  $\arrow{F_{Y}}  \, \sqcup \, \{ (Y, r) \}$ is concentrated on oriented spanning unicycles  rooted at $r$, i.e.,  oriented subgraphs of $G$ with one outgoing edge for every vertex of $G$ and one unique oriented cycle, where $r$ is contained in that oriented cycle.
Each unicycle $\Ul$  is picked with probability proportional to 
the product of the weight of its edges.
 This implies that 
$\arrow{F_{Y}}  \, \sqcup \, \{ (Y, r) \}$ is distributed  as $\arrow{\wsf}^+_r$, as desired.
\end{proof}

\smallskip

\begin{proof}[Proof of Theorem~{\hyperref[thmstationarity]{1.\thenumberstationarity}}]
Since $(\rel_n)_{n \geq 0}$ is a Markov chain,
it suffices to show that if $\rel_0$ is distributed as $\arrow{\wsf}^+_{\id}$, then $\rel_1$ is also distributed as   $\arrow{\wsf}^+_{\id}$.

Let $\Fl$ be the random  spanning forest of $G$ sampled from $\arrow{\wsf}_{\id}$.
Let $Y$ be a random neighbor of the identity sampled from $\mu_{\id}$ independently of $\Fl$.
For any $x \in V$, denote by $\tau_x: V(G) \to V(G)$ the network isomorphism of $(G,c)$ given by left multiplication by $x$. (A network isomorphism of $(G,c)$ is a graph isomorphism of $G$ which also preserves the conductance $c$.)

Since $\rho_0(\id)\overset{d}{=} Y$ and the RWLM satisfies  
\eqref{equation: c-stationary}, we have  $\rho_1(\id)\overset{d}{=} Y$.
By the transition rule  of RWLM (see \eqref{equation: transition rule RWLM}), we then have
$\rho_1 \overset{d}{=} \Fl  \, \sqcup \, \{ (\id,Y) \}$.
By the transition rule  of the scenery process (see \eqref{equation: transition rule scenery process}), we then have 
$\rel_1 \overset{d}{=} \tau_{Y^{-1}}(\Fl)  \, \sqcup \, \{ (Y^{-1}, \id) \}$.

Now note that 
 $Y\overset{d}{=}Y^{-1}$ since  $\mu_{\id}$ is symmetric, 
and  together with the conclusion of the previous paragraph this implies that  $\rel_1\overset{d}{=} \tau_{Y}(\Fl)  \, \sqcup \, \{ (Y, \id) \}$.
Also note that $\tau_{Y}(\Fl)$ is equal in distribution to the random spanning forest picked from $\arrow{\wsf}_{Y}$ since $\tau_Y$ is a network isomorphism of $(G,c)$.
It now follows from  Lemma~\ref{lemma: wsfplus  sum} 
that $\rel_1$ is distributed according to $\arrow{\wsf}^+_{\id}$, and the proof is complete.
\end{proof}

\bigskip

\section{Ergodic native environments}\label{section: ergodic}

In this section we prove Theorem~\ref{theorem: intro ergodic} by showing that $\arrow{\wsf}^+_{\id}$ is an ergodic native environment.
This requires tools from the ergodic theory of Markov chains,
which we quickly review in the next subsection,
and we  
 refer the reader to  \cite{HLL98} for a more detailed discussion on this subject.

\medskip

\subsection{Ergodic theory for Markov chains}\label{subsection: ergodic theory}
Let $M:=(\Omega, \Fscr,P)$ be a Markov chain,
where the  state space $\Omega$ is a  metric space, 
$\Fscr$ is the Borel $\sigma$-algebra of $\Omega$,
and  $P:\Omega\times \Fscr \to [0,1]$ is the probability transition function of this chain. 
A set $\Bscr \in \Fscr$ is \emph{invariant}  if $P(x,\Bscr)=1$ for all $x \in \Bscr$. 
A stationary distribution $\pi$ of $M$ is \emph{ergodic} if ${\pi}[\Bscr]\in \{0,1\}$ for any invariant set $\Bscr$. 

%

Let $\Omega^{\N}$ be the \emph{trajectory space} of $M$,
\[ \Omega^{\N}:=\{(\omega_i)_{i \geq 0} \mid \omega_i \in \Omega\},\]
equipped with the product $\sigma$-algebra induced by $\Fscr$.
For any $\omega \in \Omega$ we denote by $P_{\omega}$ the probability distribution on $\Omega^{\N}$ given by:
\[{P_{\omega}}[\Ascr]:= \E{1_{\Ascr}(\omega_0,\omega_1,\ldots)} \qquad  (\Ascr \in \Fscr ),\]
where 
$(\omega_n)_{n\geq 0}$ is the Markov chain $M$ with initial state $\omega_0=\omega$, and $\mathbb{E}$ is the corresponding expectation function for this chain.

\smallskip

\begin{theorem}[Pointwise ergodic theorem~{\cite[Theorem~6.1(b)]{HLL98}}]\label{theorem: pointwise ergodic theorem}
	Let $M$ be a Markov chain on a compact  metric space $(\Omega, \Fscr)$, and let  $\pi$ be  an ergodic distribution of $M$.
	Then  for 
	every $\pi$-integrable function $f:\Omega \to \R$, 
	\[ 
	\lim_{n \to \infty} \frac{1}{n}\sum_{i=0}^{n-1} f(\omega_i)= \int_{\Omega} f \, d\pi \qquad P_{\omega}\text{-a.s.}, \]
	for $\pi$-almost every $\omega \in \Omega$. \qed
\end{theorem}

\smallskip

The following lemma will be useful for checking if a given stationary distribution $\pi$ is ergodic.
For any $n\geq 1$, we denote by $P^{(n)}$ the $n$-step transition function of the Markov chain $M$.

\smallskip

\begin{lemma}\label{lemma: invariance implies time invariance}
	Let  $M:=(\Omega, \Fscr,P)$  be a Markov chain, and let  $\pi$ be  a stationary distribution of $M$.
	If $\Bscr$ is an invariant set, 
	then the set 
	\[ \Bscr':= \{ x \in \Omega \mid \exists~n \geq 1 \text{ s.t. } P^{(n)}(x,\Bscr)>0 \},  \]
	differs from $\Bscr$ by a set of $\pi$-measure zero.
\end{lemma}

\smallskip

\begin{proof}
	First note that $\Bscr \subseteq \Bscr'$ by the invariance of $\Bscr$.
	Now note that,
	for any $n \geq 1$,
	\begin{align*}
		\pi[\Bscr] &= \int_{\Omega} P^{(n)}(x, \Bscr) \ d\pi(x) \qquad \text{(by the stationarity of $\pi$)}\\
		&=\int_{\Bscr} P^{(n)}(x,\Bscr) \ d\pi(x) + \int_{\Bscr' \setminus \Bscr} P^{(n)}(x, \Bscr) \ d\pi(x) \qquad \text{(as $P^{(n)}(x, \Bscr) = 0$ for $x \notin \Bscr'$)}\\
		&= \pi[\Bscr] + \int_{\Bscr' \setminus \Bscr} P^{(n)}(x, \Bscr) \ d\pi(x) \qquad \text{(as $\Bscr$ is invariant).}
	\end{align*}
	Hence we conclude that $\int_{\Bscr' \setminus \Bscr} P^{(n)}(x, \Bscr) \ d\pi(x) = 0$ for any $n \geq 1$.
	It then follows from the definition of $\Bscr'$ that  that $\pi[\Bscr' \setminus \Bscr]=0$.
	This proves the lemma. 
\end{proof}

\medskip

\subsection{Proof of Theorem~\ref{theorem: intro ergodic} }\label{subsection: proof of ergodic native environment}

Recall  
the definition of 
scenery process $(\rel_n)_{n\geq 0}$ 
from Definition~\ref{definition: scenery process}.

\begin{definition}[Ergodic native environment]\label{definition: ergodic environment}
	Consider an RWLM  on a weighted Cayley graph that satisfies \eqref{equation: transitive} and \eqref{equation: c-stationary}.
	An \emph{ergodic environment} is a distribution on   rotor configurations of $G$
	that is an ergodic measure for the scenery process of the RWLM.
\end{definition}

We now restate Theorem~\ref{theorem: intro ergodic} from the introduction (also the main result of this section) in a slightly more general form.
Recall that the definition of probability transition functions
$p_x$ from Definition~\ref{definition: mechanism}. 
We say that the RWLM is \emph{elliptic} if,
\begin{equation}\label{equation: elliptic}
	p_x(y,y') \ > \ 0 \qquad \text{for every $x \in V$ and every $y,y' \in N(x)$}. \tag{Ell}
\end{equation}
Note that, from the RWLMs in \S\ref{rotor walks}, the Aldous-Broder walk and the $p,\!r$-rotor walk with $r<1$ are elliptic,
while $p$-rotor walk and  deterministic rotor walk are not elliptic.

\smallskip

\begin{reptheorem}{theorem: intro ergodic}\label{thmergodic}
	Consider an RWLM  on a weighted Cayley graph that satisfies \eqref{equation: transitive}, \eqref{equation: c-stationary}, and \eqref{equation: elliptic}.
Then   $\arrow{\wsf}^+_{\id}$ is an ergodic  native environment.
\end{reptheorem}

\smallskip


\smallskip

%
%

\begin{proof}
		Let $\pi= \arrow{\wsf}^+_{\id}$, and let $\Bscrl$ be  a set of rotor configurations that is invariant w.r.t.\ the scenery process.
	Recall the definition of tail event for rotor configurations (equivalently, oriented subgraphs)  from \S\ref{subsection: tail triviality}.
	It suffices to show that $\Bscrl$ differs from a tail event  by a set of $\pi$-measure zero,
	as it will then follow from
	the tail triviality of $\pi$ (Lemma~\ref{lemma: wsfp is tail trivial}) that $\pi[\Bscrl] \in \{0,1\}$.
	
	Let $\Rot(G)$ denote the set of rotor configurations of $G$.
	We write 
	\[\Cscrl := \{\rho \in \Rot(G)  \mid   \exists~\rho' \in \Bscrl \text{ s.t. } \rho \text{ and } \rho' \text{ differ at finitely many vertices}\}.\]
	Note that $\Cscrl$ is  a tail event that contains $\Bscrl$.
	It then suffices to show that $\pi[\Cscrl \setminus \Bscrl]=0$.

	Let $\rho$ be any rotor configuration in $\Cscrl$. Then  there exists $\rho' \in \Bscrl$ such that $\rho'$ differs from $\rho$ at finitely many vertices. Let $\langle x_0,\ldots,x_n\rangle $ be a directed walk in $G$ that starts at $o$ and such that $\{x_0,\ldots,x_{n-1}\}$ contains all the vertices for which $\rho$ and $\rho'$ differ.

	For each $i \in \{1,\ldots,n\}$, define $\rho_i$ to be the rotor configuration at the $i$-th step of the RWLM if the initial walker-and-rotor configuration is $(x_0,\rho)$ and the trajectory of the walker for the first $i$ steps is given by $\langle x_0,\ldots,x_i\rangle$. 
	That is, these rotor configurations are given by the recursive definition
	\[\rho_{i+1}(x):=\begin{cases}x_{i+1} & \text{if } x = x_i; \\ \rho_i(x) & \text{otherwise.} \end{cases}\]
	Define $\rho'_i$ in a similar manner, but with $(x_0,\rho')$ as the initial walker-and-rotor configuration.
	Note that $\rho_n = \rho'_n$ since the walker of the RWLM that follows the directed walk $\langle x_0,\ldots,x_n\rangle$ would have visited and changed the rotors at all vertices for which $\rho$ and $\rho'$ differ;
	see Figure~\ref{figure: elliptic path}.
	
	\begin{figure}[t!]
		\begin{tabular}{c @{\hskip 40 pt} c  @{\hskip 40 pt} c  @{\hskip 40 pt} c }
			\includegraphics[scale=0.8]{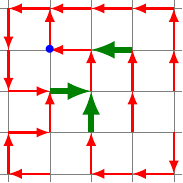} & 
			\includegraphics[scale=0.8]{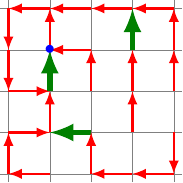} &
			\includegraphics[scale=0.8]{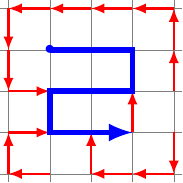}  &
			\includegraphics[scale=0.8]{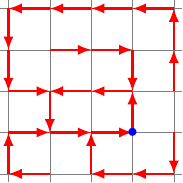} \\
			(a) & (b) & (c) & (d)
		\end{tabular}
		\caption{(a) and (b) Two rotor configurations that differ at finitely many vertices.
			The rotors at which they differ are drawn oversized in green.
			(c) The trajectory  (drawn in blue) taken by the walker that visits every green rotor. 
			(d) The final rotor configuration of the RWLM at the end of this process, which is the same regardless of whether the initial configuration is (a) or (b).
		}
		\label{figure: elliptic path}
	\end{figure}

	Write $\rho'':=\tau_{x_n^{-1}}(\rho_n)=\tau_{x_n^{-1}}(\rho_n')$.
	Note that $\rho''$ is the rotor configuration at the $n$-th step of the scenery process if the walker of the RWLM follows $\langle x_0,\ldots,x_n\rangle$ and the initial rotor configuration is $\rho$ or $\rho'$ (recall that $\tau_x$ is the network isomorphism of $(G,c)$ given by left multiplication by $x$). 
	In particular,  the probability to   transition from $\rho$ to $\rho''$ in $n$ steps of the scenery process satisfies the following inequality:
	\[P^{(n)}(\rho,\rho'') \geq \prod_{i=0}^{n-1} p_{x_i}(\rho_i(x_i),x_{i+1})>0,\]
	where the strict inequality is due to \eqref{equation: elliptic}.
	Note that, by the same argument, we also have
	$P^{(n)}(\rho',\rho'')>0$.
	
	Since $\rho' \in \Bscrl$ and $P^{(n)}(\rho',\rho'')>0$,
	we have $\rho''\in \Bscrl$ by the invariance of $\Bscrl$.
	This  implies that $\rho$ can transition into $\Bscrl$ in $n$ steps of the scenery process with positive probability, as
	\[ P^{(n)}(\rho,\Bscrl)\geq P^{(n)}(\rho,\rho'') >0.\]
	
	As the choice of $\rho \in \Cscrl$ is arbitrary, we have from the argument above that:
	\[ \Cscrl \subseteq \{ \rho \in \Rot(G) \mid \exists~n \geq 1 \text{ s.t. } P^{(n)}(\rho,\Bscrl)>0 \}.  \]
	Since $\pi$ is a stationary distribution of the scenery process (Theorem~{\hyperref[thmstationarity]{1.\thenumberstationarity}}),
	it then follows from 
	 Lemma~\ref{lemma: invariance implies time invariance} that the set on the right side of the equation differs from $\Bscrl$ by a set of $\pi$-measure zero.
	Hence we conclude that
	$\pi[\Cscrl \setminus \Bscrl]=0$, as desired.
\end{proof}



%


\bigskip

\section{Functional CLT for RWLM}\label{SLLN}

In this section we present the proof of Theorem~\ref{theorem: intro scaling limit}.
An electrical network  $(G,c)$ is a \emph{weighted lattice graph} in $\R^d$ if $G$ is weighted Cayley graph such that $V(G)$ is a 
 subgroup of $\R^d$ with vector addition as the group operation. 
In this section we will assume that $G$ is a weighted lattice graph,
and that the walker is initially located at the origin $\nol:=(0,\ldots,0)$,
 unless stated otherwise.

We now restate Theorem~\ref{theorem: intro scaling limit} from the introduction in a slightly more general form.
Recall the definition the measure $\mu_x$ from \eqref{equation: mu}.
We denote by $\Gamma$ the matrix
 \begin{align*}
	\Gamma:= \sum_{\yb \in \Nbh(\nol)}  \mu_{\nol}(\yb)\, \yb \, \yb^\top.
\end{align*}
Recall that $X_n$ and $\rho_n$ $(n\geq 0)$ denotes the location of the walker and the rotor configuration at the $n$-th step of the RWLM, respectively, and  
 that $D_{\R^d}[0,\infty)$ denotes the Skorohod space  of $\R^d$-valued c\`{a}dl\`{a}g paths on $[0,\infty)$.

\smallskip

\begin{reptheorem}{theorem: intro scaling limit}
	\label{thmscaling limit}
	Consider an RWLM on a weighted lattice graph in $\R^d$  that is 
	\eqref{equation: transitive}, \eqref{equation: c-stationary}, and \eqref{equation: martingale}.
	Suppose that the initial environment $\pi$ is an ergodic native environment.
	Then, for   almost every environment sampled from $\pi$, 
the scaled walk $(\frac{1}{\sqrt{n}} X_{\lfloor nt\rfloor } )_{t \geq 0}$ converges weakly on $D_{\R^d}[0,\infty)$ to a Brownian motion with diffusion matrix $\Gamma$.
\end{reptheorem}

\smallskip

As a consequence of Theorem~\hyperref[thmscaling limit]{1.\thenumberscalinglimit}, 
the $p,\!r$-rotor walk on $\Z^d$ (Example~\ref{example: p,r-rotor walk}) with constant conductance, with $p=\frac{1}{2}$ and $r<1$,
and with  $\arrow{\wsf}^+_{\id}$ as the initial environment, converges weakly  on  $D_{\R^d}[0,\infty)$ to a Brownian motion with diffusion matrix $\frac{1}{d}I_{d}$.

\smallskip

 \begin{proof}[Proof of Theorem~{\hyperref[thmscaling limit]{1.\thenumberscalinglimit}}]
 Let $V_n:=X_{n+1}-X_n$ and   $\Fscr_n := \sigma(X_0,\dots,X_n,\rho_0,\dots,\rho_n)$.	
 It suffices to  verify that all conditions in Theorem~\ref{CLT} are satisfied. 
By using the same argument as in the proof of  
 Proposition~\hyperref[propscaling limit]{1.\thenumberprops}, 
we have that 
 $(X_n)_{n\geq 0}$
 is a square-integrable martingale process (as a consequence of \eqref{equation: martingale}), and that \eqref{CLT2} is satisfied.
 We omit the details for brevity.

We now verify \eqref{CLT1}.
 Let $i\geq 0$. 
It follows from Definition~\ref{definition: random walk with local memory} and \eqref{equation: transitive} that 
\begin{align*}
 V_i=& X_{i+1}- X_{i}
= \sum_{\yb \in  \Nbh(\nol)} \satubb \{\rho_{i}(X_i)-X_i=\yb \}   \,  Y_{\yb,i},
\end{align*}
where 
  $Y_{\yb,i}$ is a random variable on neighbors of the origin sampled from $p_{\nol}(\yb, \cdot) $ independently of $\Fscr_i$.
Then, for any $n\geq 0$:
\begin{align}\label{eqreferee 1}
 \frac{1}{n}\sum_{i=0}^{n-1} \Cond{V_i V_i^\top}{\Fscr_i}
=& \sum_{\yb \in  \Nbh(\nol)} \left(   \frac{1}{n}\sum_{i=0}^{n-1}  \satubb \{\rho_{i}(X_i)-X_i=\yb \}\right)\, \E{Y_{\yb,0}\, Y_{\yb,0}^\top}.
\end{align}
Here we have used the fact that $Y_{\yb,i}$ has the same law as ${Y_{\yb,0}}$ for all $i$.

We now show that, for every 
 $\yb \in N(\nol)$, 
\begin{equation}\label{eqreferee 2}
	 \lim_{n \to \infty} \frac{1}{n} \sum_{i=0}^{n-1} \satubb \{\rho_{i}(X_i)-X_i=\yb \} \ = \  \mu_{\nol}(\yb).
\end{equation}
	Fix an ordering  $x_1,x_2,\ldots$ of $V(G)$. 
Note that the set of rotor configurations $\Rot(G)$ is a compact metric space with the metric
$d(\rho_1,\rho_2):=\sum_{i=1}^\infty\frac{1}{2^i} \satubb \{\rho_1(x_i)\neq \rho_2(x_i)\}$.
It is straightforward to check that $\Fscrl(G)$ (from \S\ref{subsection: wsf oriented toward a root}) restricted to $\Rot(G)$
is  the Borel $\sigma$-algebra corresponding to this metric. 
Hence all conditions of Theorem~\ref{theorem: pointwise ergodic theorem} are satisfied, and 
\eqref{eqreferee 2} now follows by applying 
Theorem~\ref{theorem: pointwise ergodic theorem} 
to the function $f:\Rot(G) \to \R$ given by
$f(\rel):=\satubb \{ \rel(\nol)=\yb  \}$.


Plugging \eqref{eqreferee 2} into \eqref{eqreferee 1}, we get
\begin{align*}
\lim_{n \to \infty} \frac{1}{n}\sum_{i=0}^{n-1} \Cond{V_i V_i^\top}{\Fscr_i}
=& \sum_{\yb \in \Nbh(\nol) } \mu_{\nol}(\yb) \, \E{Y_{\yb,0}\, Y_{\yb,0}^\top} \\
=&\sum_{\yb \in \Nbh(\nol) }  \mu_{\nol}(\yb)\sum_{\yb' \in \Nbh(\nol)}  \, p_{\nol}(\yb,\yb') \, \yb'\, \yb'^\top.
\end{align*}
Since $\mu_{\nol}$ is a stationary distribution of the mechanism at $\nol$ by  \eqref{equation: c-stationary},
it then follows that:
\begin{align*}
\lim_{n \to \infty} \frac{1}{n}\sum_{i=0}^{n-1} \Cond{V_i V_i^\top}{\Fscr_i}
=&\sum_{\yb' \in \Nbh(\nol) }  \mu_{\nol}(\yb') \, \yb'\, \yb'^\top=\Gamma.
\end{align*}
 Hence \eqref{CLT1} is verified, and the proof is complete.
 \end{proof}
 
%
 
 \bigskip

\section{Concluding remarks}\label{section: further questions}

We conclude with a few natural questions.


\subsection{} 
Theorem~\ref{theorem: intro scaling limit}
allows us to derive a functional CLT, but only when the initial environment is an ergodic native environment.
Does the conclusion of Theorem~\ref{theorem: intro scaling limit} still hold for other initial environments?
We believe that the answer to this question is positive for the iid initial environment, 
and simulations  suggests that there should be no quantitative difference between 
iid initial environment and wired spanning forest plus one edge environment eventually.
\begin{problem}
	Consider an RWLM  on a simple Cayley graph that is transitive, uniform, and elliptic.
	Let $(\rel_n)_{n \geq 0}$ be the scenery process of the RWLM with iid initial environment.
	Show that $\rel_n$ converges weakly to $\arrow{\mathsf{WUSF}}^+$, i.e.,  for every edge $\{x_1, y_1\},\ldots, \{x_m,y_m\}$ of $G$,
	\[  \mathbb{P} \big[ \rel_n(x_1)=y_1, \ldots, \rel_n(x_m)=y_m \big]   \ \overset{ n \to \infty}{ \longrightarrow} \ 
	\mathbb{P} \big[ (x_1, y_1), \ldots, (x_m,y_m)  \in \Ul \big],\]
	where  $\Ul$ is a random subgraph sampled from $\arrow{\mathsf{WUSF}}^+$.
\end{problem}

\subsection{}\label{subsection: dropping ellipticity} 
%
%
Can we drop the ellipticity assumption from Theorem~\ref{theorem: intro ergodic} and Theorem~\ref{theorem: intro scaling limit}?
In particular, a positive answer to this question will give us a scaling limit result for  $p$-rotor walk on $\Z^d$ $(d\geq 2)$ when $p=\frac{1}{2}$, which will be consistent with the simulation results in Figure~\ref{figure: simple random walk vs rotor walk}.


%


\subsection{}
 An RWLM is \emph{recurrent} if every vertex is visited infinitely often by the walker a.s.\ and is \emph{transient} otherwise.
Note  that every   $d$-dimensional RWLM  on $\Z^d$ satisfying conditions  in 
 Theorem~\ref{theorem: intro scaling limit}
is  transient if $d\geq 3$ (as the transience of the scaling limit implies the transience of the original walk).
Is it true that these RWLMs are recurrent if $d=2$? 
We remark that a partial answer to this question has been given in \cite{Cha20} (the sequel to this paper), namely for  the `H'--`V' walk on $\Z^2$ with i.i.d. initial environment and $p=\frac{1}{2}$, and remains open for other values of $p$.


%

%

 \section*{Acknowledgement}
We would like to thank Elena Kosygina, Yuval Peres and Ofer Zeitouni for inspiring discussions with the first author and the third author.
We would also like to thank Timo Sepp{\"a}l{\"a}inen for pointing us to references for Theorem~\ref{CLT}.
The first author would also like to thank Igor Pak for writing advice.
Last but not the least, we would like to thank the anonymous referees for the insightful comments that substantially improve the readability of the paper and for additional references.

\appendix

\section{Random walks with hidden local memory}\label{section: hidden random walks}
In this section we present a more general version of random walk with local memory
  inspired by  hidden Markov chains. We refer to  \cite{Bil06} for a more detailed discussion on  hidden Markov chains.

For each $x \in V$, a \emph{hidden mechanism} at $x$ is a Markov chain $M_x$ with finite state space $S_x$ and probability transition function $p_x(\cdot,\cdot)$.
A \emph{jump rule} is a map  $f_x:S_x \to \mathcal{P}(\Nbh(x))$  from $S_x$ to the set of probability distributions on the set of neighbors of $x$. 
 A {\em hidden state configuration} is a map $\kappa: V \to \sqcup_{x \in V} S_x$ such that $\kappa(x) \in S_x$ for all $x \in V$.

\smallskip

\begin{definition}[Random walk with hidden local memory]\label{def: hidden walk}
A \emph{random walk with hidden local memory}, or RWHLM for short, is a sequence  
$(X_n, \rho_n, \kappa_n)_{n \geq 0}$ satisfying the following transition rules:
\begin{enumerate}
\item $
 \kappa_{n+1}(x):=\begin{cases} K_n & \text{if } x=X_n;\\
\kappa_{n}(x) & \text{if } x\neq X_n. \end{cases}$;
 
\item $ \rho_{n+1}(x):=\begin{cases} Y_n & \text{if } x=X_n;\\
\rho_{n}(x) & \text{if } x\neq X_n, \end{cases}$
\item   $X_{n+1}:=Y_n$,
\end{enumerate}
where $K_n$ is a random element of $S_{X_n}$ sampled from $p_{X_n}(\kappa_{n}(X_n), \cdot)$ 
independent of the past, and $Y_n$ is  a random neighbor of $x$  sampled from $f_{X_n}(K_n)$ independent of the past. 
\end{definition}

\smallskip

Described in words, 
at each time step  (i) the walker  first updates the hidden state 
of its current location using the given hidden mechanism.
 Then, (ii) the walker  updates the rotor  of its current location  by sampling the new rotor   from the probability distribution corresponding  to the new hidden state.
  Finally, (iii) the walker  travels to the vertex specified by  the new rotor.

\smallskip

\begin{example}[Hidden triangular walk]\label{ex: hidden triangular lattice} Let $G$ be the triangular lattice. 
For each $\xb \in V$, the hidden mechanism at $\xb \in V$
has the following state space and transition probability:
\begin{align*}
S_{\xb}:=\{ s_1,s_2,s_3 \};\qquad 
p_\xb :=\begin{bmatrix}
0 & \frac{1}{2} & \frac{1}{2}\\
0 & 0 & 1\\
1 &0 &0
\end{bmatrix}.
\end{align*}
That is, $s_1$ transitions to either $s_2$ or $s_3$ with equal probability, $s_2$ transitions to $s_3$ with probability 1, and $s_3$ transitions to $s_1$ with probability 1.

We now describe the  jump rule $f_\xb$.
Let $N_1 \sqcup N_2$ be the partition of  the neighbors $\Nbh(\xb)$ of $\xb$ given by:
\begin{align*}
N_1:= \xb +\left\{\begin{pmatrix}
1\\ 0
\end{pmatrix}, \frac{1}{2}\begin{pmatrix}
-{1}  \\ {\sqrt{3}}
\end{pmatrix} , \frac{1}{2}\begin{pmatrix}
-1 \\ {-\sqrt{3}}
\end{pmatrix}\right\};	\qquad N_2:=\xb+ \left\{\begin{pmatrix}
	-1\\0
	\end{pmatrix},\frac{1}{2}\begin{pmatrix}
	1 \\ {\sqrt{3}}
	\end{pmatrix},   \frac{1}{2}\begin{pmatrix}
	1\\ -\sqrt{3}
	\end{pmatrix}\right\}.
\end{align*}
  The distribution $f_\xb(s_1)$  is then given by the uniform distribution on $N_1$,
while  $f_\xb(s_2)$ and $f_\xb(s_3)$ are   the uniform distribution on $N_2$.

\begin{figure}[t!]
\includegraphics[scale=0.8]{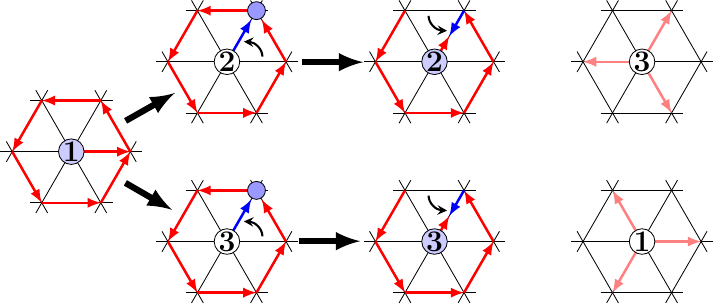} 
\caption{Two instances of a two-step hidden triangular walk with the same walker's trajectory and rotor configurations.
 The number at the origin records the hidden state of the origin.
The pictures at the right side illustrate the future hidden state of the origin and the arrows point to  (possible) future locations of the walker.
}
 \label{figure: hidden triangular walk}
\end{figure}

Without knowing the hidden states, an outside observer will not be able to predict the future dynamics of this RWHLM even while knowing the past and present location of the walker and rotor configuration, as illustrated in Figure~\ref{figure: hidden triangular walk}. 
\end{example}

\smallskip

Note that  a non-hidden RWLM is a  special case of  RWHLM,
with  $S_x$ ($x \in V$)  being  the set of neighbors of $x$ and with 
$f_x(y)$ $(y \in \Nbh(x))$ being  the probability distribution  concentrated on $y$.
On the other hand, 
every RWHLM on a simple graph $G$ can be emulated by a non-hidden RWLM  on a  larger graph (with multiple edges) in the following manner.

Let $G^\times$ be the undirected graph 
with 
vertex set $V(G)$
and with an edge   incident to  $x$ and $y$ in $G^\times$ for each $\{x,y\}\in E(G)$ and each hidden state $s \in S_x$ of the RWHLM. 
Such an edge is labeled $e({x,y,s})$.

For any $x \in V(G^\times)$, the  mechanism of this RWLM on $x$  is the Markov chain with state space the set of edges incident to $x$ in $G^\times$ (instead  of the set of neighbors of $x$), and with probability transition function
\[ p_x^\times(e({x,y,s}),e({x,y',s'})) := p_{x}(s,s')\,   (f_x(s'))(y'), \]
where $p_x$ and $f_x$ are the probability transition function and the jump rule for the RWHLM, respectively.

This RWLM on $G^\times$ emulates the RWHLM on $G$ in the following sense.
Let $(X_n,\rho_n,\kappa_n)_{n \geq 0}$ be an RWHLM on $G$.
Start an RWLM  $(X_n^\times, \rho_n^\times)_{n \geq 0}$ on $G^\times$ with the following initial configuration:
\begin{align*}
X_0^\times&:= X_0; \qquad \rho^\times_0(x):= e({x, \rho_0(x),   \kappa_0(x))} \quad (x \in V).
\end{align*} 
Then $(X_n, \rho_n)_{n \geq 0}$ is equal in distribution to 
 $(X_n^\times , h(\rho_n^\times))_{n \geq 0}$,
 where $h(\rho_n^\times)$ is the rotor configuration of $G$ 
 given by $h(\rho_n^\times)(x):=y$ if $\rho_n^\times(x)=e({x,y,s})$ for some $s \in S_x$.
 
As a consequence of this reduction, we can convert the hidden triangular walk from Example~\ref{ex: hidden triangular lattice} to a non-hidden random walk with local memory, and then apply a version of Proposition~\hyperref[propscaling limit]{1.\thenumberprops} for non-simple graphs to conclude that  the scaling limit of  this hidden triangular walk is a Brownian motion in $\R^2$.


\end{document}